\documentclass[11pt]{article}
\usepackage{hyperref}
\usepackage[T1]{fontenc}
\usepackage{amsfonts}
\usepackage[english]{babel}
\usepackage{a4wide,times}

\usepackage[latin1]{inputenc}
\usepackage{amssymb}
\usepackage{graphicx}
\usepackage{epsfig}
\usepackage{amsbsy}
\usepackage{verbatim}
\usepackage{color}
\usepackage{mathrsfs}
\usepackage{makeidx}
\usepackage{amsmath}
\usepackage{amsthm}

\theoremstyle{plain}
\newtheorem{thm}{Theorem}[section]

\newtheorem{lem}[thm]{Lemma}
\newtheorem{prop}[thm]{Proposition}

\newcommand{\argmin}{\arg\!\min}

\theoremstyle{definition}
\newtheorem{defn}{Definition}[section]
\usepackage{dsfont}

\theoremstyle{remark}
\newtheorem{rem}{\bf Remark}[section]
\theoremstyle{remark}

\newtheorem{com*}{\bf Comment}

\usepackage{fancyhdr}

\makeatletter
\def \newequation#1#2{
   \@definecounter{#1}
   \@namedef{the#1}{\hbox{#2}}
   \@namedef{#1}{$$\refstepcounter{#1}}
   \@namedef{end#1}{
      \eqno \csname the#1\endcsname $$\global\@ignoretrue
      }
}
\makeatother
\newequation{E1}{($E_{b,\sigma,W}$)}
   
\makeatletter
\def \newequation#1#2{
   \@definecounter{#1}
   \@namedef{the#1}{\hbox{#2}}
   \@namedef{#1}{$$\refstepcounter{#1}}
   \@namedef{end#1}{
      \eqno \csname the#1\endcsname $$\global\@ignoretrue
      }
   }
\makeatother
\newequation{hyp3}{($\mathcal{H}_{b,\sigma}$)}

\makeatletter
\def \newequation#1#2{
   \@definecounter{#1}
   \@namedef{the#1}{\hbox{#2}}
   \@namedef{#1}{$$\refstepcounter{#1}}
   \@namedef{end#1}{
      \eqno \csname the#1\endcsname $$\global\@ignoretrue
      }
   }
\makeatother
\newequation{shleg}{(ShLeg)}

\makeatletter
\def \newequation#1#2{
   \@definecounter{#1}
   \@namedef{the#1}{\hbox{#2}}
   \@namedef{#1}{$$\refstepcounter{#1}}
   \@namedef{end#1}{
      \eqno \csname the#1\endcsname $$\global\@ignoretrue
      }
   }
\makeatother
\newequation{kl}{(KL)}

\makeatletter
\def \newequation#1#2{
   \@definecounter{#1}
   \@namedef{the#1}{\hbox{#2}}
   \@namedef{#1}{$$\refstepcounter{#1}}
   \@namedef{end#1}{
      \eqno \csname the#1\endcsname $$\global\@ignoretrue
      }
   }
\makeatother
\newequation{haar}{(Haar)}

\title{Recursive  marginal quantization  of  the Euler  scheme of a diffusion process}

\author{ 
{\sc  Gilles Pag\`es} \thanks{Laboratoire de Probabilit\'es et Mod\`eles al\'eatoires (LPMA), UPMC-Sorbonne Universit\'e, UMR  CNRS 7599, case 188, 4, pl. Jussieu, F-75252 Paris Cedex 5, France.
E-mail: {\tt gilles.pages@upmc.fr} } \ \ \ 
{\sc  Abass Sagna} \thanks{ENSIIE \& Laboratoire de Math\'ematiques et Mod\'elisation d'Evry (LaMME), Universit\'e  d'Evry Val-d'Essonne,  UMR CNRS  8071,    23 Boulevard de France, 91037 Evry. E-mail: {\tt abass.sagna@ensiie.fr}. }\ \ \
\thanks{The  first  author  benefited from the support of the Chaire ``Risques financiers'',  a joint initiative of \'Ecole Polytechnique, ENPC-ParisTech and UPMC, under the aegis of the Fondation du Risque. The second author  benefited from the support of the Chaire ``Markets in Transition'', under the aegis of Louis Bachelier Laboratory, a joint initiative of \'Ecole polytechnique, Universit\'e d'\'Evry Val d'Essonne and  F\'ed\'eration Bancaire Fran\c{c}aise.} 
}

\date{}

\begin{document}

\maketitle

\begin{abstract}
We propose  a new approach to quantize the marginals of the discrete Euler diffusion  process. The method is built recursively  and involves the  conditional distribution  of the marginals of the discrete Euler process.  Analytically,  the method raises several  questions like  the analysis  of the induced quadratic quantization error between  the marginals of the  Euler process and the proposed quantizations. We  show  in particular that  at  every  discretization step $t_k$ of the Euler scheme, this error is bounded by the  cumulative   quantization errors  induced by   the Euler operator,  from times $t_0=0$  to time $t_k$.  For numerics, we restrict our analysis to  the one dimensional setting and  show how to compute the optimal grids using a Newton-Raphson  algorithm. We   then  propose a closed formula  for the  companion weights and the  transition probabilities associated to the proposed quantizations.  This allows us to quantize  in particular  diffusion processes  in  local volatility models  by  reducing  dramatically the computational complexity of the search of optimal quantizers  while increasing their computational precision with respect to the  algorithms  commonly proposed  in this framework. Numerical tests are carried out  for  the Brownian motion and for the pricing of European options in a local volatility model.  A comparison with  the Monte Carlo simulations shows that the proposed method may sometimes  be more  efficient (w.r.t.  both computational  precision and time complexity) than   the Monte Carlo method.
\end{abstract}

\section{Introduction}
Optimal quantization method  appears first in~\cite{She}  where the author studies in particular the optimal quantization problem for the uniform distribution. It has become an important field of information theory since the early $1940$'s.  A common use of quantization is the conversion of a continuous signal  into a discrete signal that assumes only a finite number of values. 
  
  Since then, optimal quantization has been   applied in many fields like in Signal Processing, in Data Analysis,  in Computer Sciences and recently in Numerical Probability following the work~\cite{Pag}. Its application to Numerical Probability relies on the possibility to discretize   a random vector $X$ taking infinitely many values by   a discrete random vector $\widehat X$ valued in a set of finite cardinality. This allows to approximate  either expectations of the form $\mathds Ef(X)$  or, more significantly, some conditional expectations  like  $\mathds E(f(X) \vert Y)$ (by quantizing both random variables $X$ and $Y$).   Optimal quantization is used to solve  problems emerging  in Quantitative Finance as   optimal stopping problems  (see~\cite{BalPag, BalPagPri}), the pricing of swing options  (see~\cite{BarBouPag}),  stochastic control problems  (see~\cite{CorPhaRun, PagPhaPri}), nonlinear filtering  problems (see e.g.~\cite{PagPha, PhaRunSel,CalSag,ProSag}), the pricing of barrier options (see~\cite{Sag}).  
  
  In  Quantitative Finance, several problems of interest amount to  estimate  quantities of the form
  \begin{equation}  \label{EqPriceEuroLikeOptionsC}
    \mathbb E \big[f( X_{T}) \big], \quad T>0, 
    \end{equation}
for a given Borel function $f: \mathbb R^d \rightarrow \mathbb R$, or involving terms like 
\begin{equation}  \label{EqNLfilterLikeOptionsC}
 \mathbb E \big[f(X_t) \vert   X_{s} = x \big] , \quad  0<s <t,
 \end{equation}
where $(X_t)_{t\in [0,T]}$  is   a stochastic diffusion process,  solution to  the stochastic differential equation  
\begin{equation} \label{EqSignalIntro}
 X_t = X_0+\int_0^t b(s,X_s) ds + \int_0^t\sigma(s,X_s) dW_s,   
 \end{equation}
where  $W$ is a  standard  $q$-dimensional  Brownian motion starting  at $0$ independent of the $\mathbb R^d$-valued random vector $X_0$, both defined on a probability space $(\Omega,{\cal A}, \mathbb P)$. The functions $b:[0,T] \times \mathbb R^d \rightarrow \mathbb R^d$ and the matrix-valued  diffusion coefficient function  $\sigma:[0,T] \times \mathbb R^d \rightarrow \mathbb R^{d \times q}$ are Borel measurable and satisfy some appropriate Lipschitz continuity  and linear growth conditions to ensure  the existence of  a strong solution of the stochastic differential equation (see~\eqref{LipEDS} and ~\eqref{BoundEDS} further on).  Since in general the solution of (\ref{EqSignalIntro}) is not explicit, we have  first to approximate the continuous paths of the process $(X_{t})_{t\in [0,T]}$ by  a discretization  scheme, typically,  the Euler  scheme. Given the  (regular) time discretization mesh  $t_k = k  \Delta$, $k=0,\ldots, n$, $\Delta = T/n$, the  "discrete time  Euler scheme"  $(\bar  X_{t_k}^n)_{k=0,\ldots,n}$, associated to  $(X_t)_{t\in [0,T]}$  is recursively  defined  by 
$$\bar X_{t_{k+1}}^n= \bar X_{t_k}^n + b(t_k,\bar X_{t_k}^n) \Delta  + \sigma(t_k,\bar X_{t_k}^n) (W_{t_{k+1}} - W_{t_k}), \quad  \bar X_0^n = X_0.$$
Then, once we have access  to the discretization scheme of the stochastic process $(X_t)_{t\in [0,T]}$,  the quantities (\ref{EqPriceEuroLikeOptionsC}) and (\ref{EqNLfilterLikeOptionsC}) can be  estimated by
  \begin{equation}  \label{EqPriceEuroLikeOptions}
    \mathbb E \big[f(\bar X_{T}^n)\big] 
    \end{equation}
and
\begin{equation}  \label{EqNLfilterLikeOptions}
 \mathbb E \big[f(\bar X_{t_{k+1}}^n) \vert  \bar X_{t_k}^n = x \big] \quad \textrm{when  } t=t_{k+1} \textrm{ and }  s=t_k. 
 \end{equation}

\begin{rem}  \label{RemConverEuler} $(a)$  Note  that, if $b$ and $\sigma$ both have linear growth in $x$ uniformly in $t\!\in [0,T]$, then any strong solution to~\eqref{EqSignalIntro}  (if any) and the Euler scheme  satisfy for every $p \!\in (0,+\infty)$,  
\begin{equation} \label{EqExpSupXtilde}
\sup_{n\ge 1}\mathbb E \Big( \max_{k=0,\ldots,n} \vert  \bar X_{t_k}^n   \vert^p \Big) + \mathbb E \Big( \sup_{t\in [0,T]} \vert  X_{t}   \vert^p \Big) < +\!\infty.
\end{equation}
Hence, the above quantities   \eqref{EqPriceEuroLikeOptionsC}, \eqref{EqNLfilterLikeOptionsC}, \eqref{EqPriceEuroLikeOptions}, \eqref{EqNLfilterLikeOptions}, are well defined as soon as $f$ has polynomial growth (when $X$ is well defined). 

\smallskip
\noindent $(b)$  When $b$ and $\sigma$ are smooth with bounded derivatives, say $C^{4}_b$, the following weak error result holds for smooth functions $f$ with polynomial growth (or bounded Borel functions under hypo-ellipticity  assumptions  on $\sigma$,  see  $e.g.$~\cite{BalTal,TalTub}),  the estimation of $  \mathbb E\big[f( X_T)\big]$ by $\mathbb E\big[f(\bar  X_T^n)\big]$ induces the following  weak  error:  
$$ 
\big\vert  \mathbb E f(X_T)   - \mathbb E f(\bar X_T^n) \big\vert  \leq \frac{C}{n}  
$$
with  $C=C_{b,  \sigma, T} >0$  and where $n$ is the number of time discretization steps.
\end{rem}
From now on, we will drop the exponent $n$ and will denote  $\bar X$  instead of $\bar X^n$. 

The estimation of quantities like  (\ref{EqPriceEuroLikeOptions}) or (\ref{EqNLfilterLikeOptions})  can  be performed by Monte Carlo simulations since the Euler scheme is simulable. Nevertheless, an alternative to the Monte Carlo method can  be  to use cubature formulas   produced by an optimal  quantization approximation method, especially  in small or medium dimension ($d \leq 4$  theoretically but, in practice, it may remain  competitive with respect to the Monte Carlo method up to dimension  $d=10$, see~\cite{PagPri}).

A quantization of an $\mathds R^d$-valued random vector $Y$ induced by a grid $\Gamma \subset \mathds R^d$   is a random vector  $\widehat Y^{\Gamma} = \pi(Y)$, where $q: \mathds R^d \rightarrow \Gamma$ is a Borel function. Optimal quantization consists in specifying both $\pi$ and $\Gamma$ to minimize $\Vert  Y - \widehat Y^{\Gamma} \Vert_r$ for a given $r \in (0,+\infty)$ (we talk about quadratic optimal quantization when $r=2$).
 
 Now, suppose that  we have  access to the optimal quantization or to some ``good'' (in a sense to be specified further on)   quantizations $\big(\widehat X_{t_k}^{\Gamma_k}\big)_{k=0,\ldots,n}$  (we will sometimes denote   $\widehat X_{t_k}$ instead of $\widehat X_{t_k}^{\Gamma_k}$ to simplify notations) of the process $(\bar X_{t_k})_k$  on the grids  $\Gamma_k:=\Gamma_k^{(N_k)} = \{x^{k}_1,\ldots,x^{k}_{N_k}  \}$ of size $N_k$ (which is also  called an $N_k$-quantizer), for $k=0,\ldots,n$. 
 
  If $X_0$ is random, then  we suppose that its distribution  can  be quantized. In this case $\widehat X_0^{\Gamma_0}$ will be the optimal quantization of $X_0$ of size $N_0$. If $X_0 =x_ 0$ is deterministic then   $\Gamma_0 = \{ x_0 \}$ and $\widehat X_0^{\Gamma_0}=x_0$. 
  
  Suppose as well  that we have access to or have computed (offline)  the associated weights $\mathbb P(\widehat X_{t_k}^{\Gamma_k}  =  x^{k}_i)$, $i=1,\ldots,N_k$,  $k=0,\ldots,n$ (which are the distributions of the $\widehat X_{t_k}^{\Gamma_k}$), and  the transition probabilities $\widehat p_k^{ij}  = \mathbb P(\widehat X_{t_{k+1}}^{\Gamma_{k+1}}  =  x^{k+1}_j \vert  \widehat X_{t_k}^{\Gamma_k}  =  x^{k}_i)$ for every $k=0,\ldots,n-1$, $i =1, \ldots,N_k$, $j=1, \ldots, N_{k+1}$ (in other words, the  conditional distributions ${\cal L}(\widehat X_{t_{k+1}}^{\Gamma_{k+1}} \vert \widehat X_{t_k}^{\Gamma_k})$).  Then,  using optimal quantization cubature formula, the expressions (\ref{EqPriceEuroLikeOptions}) and (\ref{EqNLfilterLikeOptions})    can be   estimated by 
$$  \mathbb E \big [f(\widehat X_{t_n}^{\Gamma_n} ) \big]  = \sum_{i=1}^{N_n} f(x^{n}_i)  \, \mathbb P \big(\widehat X_{t_n}^{\Gamma_n}   =  x^{N_n}_i \big),$$
since  $t_n =T$ and 
$$ \mathbb E \big [f(\widehat X_{t_{k+1}}^{\Gamma_{k+1}} ) \vert   \widehat X_{t_k}^{\Gamma_k} =  x^{k}_i   \big]  = \sum_{j=1}^{N_{k+1}} \widehat p_k^{ij} f(x^{k+1 }_j) ,$$
respectively.   The  remaining question to be solved is then  to know how to get the optimal  or at least ``good'' grids $\Gamma_k$, for  every $k=0, \ldots,n$, their associated weights and transition probabilities.   In a general framework, as soon as  the stochastic process $(\bar X_{t_k})_k$ (or the underlying diffusion process $(X_{t})_{t \geq 0}$) can be simulated   one may  use  zero search  stochastic gradient algorithm known as Competitive Learning Vector Quantization  (CLVQ), see e.g. ~\cite{PagPri}  or a randomized  fixed point procedure  (see e.g. ~\cite{GerGra,  PagPri03, PagYu})   to compute the (hopefully almost) optimal grids and their associated weights or transition probabilities.  In the special case of one dimension,  we may rely on the deterministic counterpart of these procedures like  fugy's algorithm or Lloyd method and, for few specific scalar distributions (the Normal distribution, the exponential distribution, the Weibull distribution, etc), to  the Newton-Raphson's  algorithm. In this last case, we can speak of fast quantization procedure since this deterministic algorithm leads to more precise estimations and is  dramatically  faster  than stochastic optimization methods.  As a typical example of implementation, quadratic optimal  quantization grids associated to  $d$-dimensional Normal distribution up to $d=10$ can  be downloaded  at the  website 
 
 \centerline{{\tt www.quantize.math-fi.com}. }

 To highlight  the usefulness   of our method, suppose for example  that  we want to price a Put option with a maturity $T$,  a  strike $K$ and a  present value $X_0$ in a 
 local volatility model where  the dynamics of the stock price process  evolves following the stochastic differential equation  (called Pseudo-CEV in~\cite{LemPag}):
\begin{equation}
dX_t = rX_t dt + \vartheta \frac{X_t^{\delta+1}}{\sqrt{1+X_t^2}} dW_t, \quad X_{0} = x_{0},  \quad  t\!\in [0,T],  
\end{equation}
where  $\delta\!\in (0,1)$, $\vartheta\!\in (0,\underline{\vartheta}]$, $\underline{\vartheta}>0$, and $r$ stands  for  the interest rate. In this situation the (unique strong) solution   $X_T$, at time $T$, is not  known analytically and  if we want to estimate the  quantity of interest: $ e^{-rT} \mathbb E(f( X_T))$,  where $f(x) := \max(K-x,0)$ is the (Lipschitz continuous) payoff function,  we have first of all  to discretize  the process $(X_{t})_{t\in [0,T]}$  as $(\bar X_{t_k})_{k=0,\ldots,n}$, with $t_n = T$  (using   e.g. the Euler scheme),  and   then     estimate   
$$ e^{-rT} \mathbb E(f(\bar  X_T)) $$
by optimal quantization.  The only way to produce  optimal grids and the associated weights in this situation  is to perform  stochastic procedures like the CLVQ  or  randomized Lloyd's procedure (see e.g.~\cite{GerGra, PagYu}), even in the one dimensional framework.   However, these methods may be very time consuming.   In this framework (as well as in the general local volatility model framework in dimension $d=1$), our approach allows us to quantize the diffusion process in the Pseudo-CEV model using the Newton-Raphson algorithm.  It is important to notice that the companion  weights and the probability transitions associated to the quantized process are obtained by a closed formula so that the method  involves by no means  Monte Carlo simulations.  
  On the other hand, a  comparison with  Monte Carlo simulation for  the pricing of European options in a local volatility model  also shows that the proposed method may  sometimes be  more  efficient (with respect to  both computational  precision and time complexity)  than   the Monte Carlo method.

\bigskip 
{\em  The recursive marginal quantization algorithm}.  Let us be more precise about  our proposed method in the general setting where  $(X_t)_{t\in [0,T]}$ is a solution  to   Equation~\eqref{EqSignalIntro}.  Our aim is in practice to compute the quadratic  optimal   quantizers $(\Gamma_k)_{0 \le k \le n}$ associated with the Euler scheme  $(\bar X_{t_k})_{0 \le k \le n}$. Such a sequence $(\Gamma_k)$  is defined for every $k=0,\ldots,n$ by 
  \begin{equation*} 
\Gamma_k \!\in  \argmin \{ \bar D_{k}(\Gamma), \Gamma \subset \mathbb R^d, {\rm card}(\Gamma) \leq N_{k} \}
\end{equation*}
where  the function $\bar D_k(\cdot)$ is the so-called  distortion function associated to $\bar X_{t_k}$, and is defined for every  $N_k$-quantizer $\Gamma_k$ by 
$$ \bar D_k(\Gamma_k) = \mathbb E \big \vert \bar X_{t_{k}}  -  \widehat{X}_{t_{k}}^{\Gamma_k}   \big\vert^2 = \mathbb E \big( {\rm dist}(\bar X_{t_{k}}, \Gamma_{k})^2 \big ),
$$
where ${\rm dist}(\bar X_{t_{k}}, \Gamma_{k})$ is the distance of $\bar X_{t_{k}}$ to $ \Gamma_{k}$ and $\vert  \cdot \vert$ is,  unless otherwise specified, the Euclidean norm in $\mathds R^d$.  At this step, there is no easy way to compute the distortion function  $\bar D_k(\cdot)$, for $k \ge 1$,  because of the form of the density function of $\bar X_{t_{k}}$. However,  by conditioning with respect to $\bar X_{t_{k-1}}$,  we can connect the  distortion function  $\bar D_{k} (\Gamma_{k})$ associated to  $\bar X_{t_{k}}$  with the  distribution of $\bar X_{t_{k-1}}$ by introducing  the Euler scheme operator  as follows:
     \begin{eqnarray}  \label{EqDistAsDistXtkIntro}
   \bar  D_{k}(\Gamma_{k}) 
    & = &  \mathbb E \big[   {\rm dist}({\cal E}_{k-1}(\bar X_{t_{k-1}},Z_{k}), \Gamma_{k})^2 \big]    
      \end{eqnarray}
   where $(Z_{k})_k $ is an $i.i.d.$ sequence of ${\cal N}(0;I_q)$-distributed  random vectors   independent from $\bar X_0$ and  for every $x\!\in \mathbb R^d$, the  Euler operator  ${\cal E}_{k-1}(x,Z_{k})$  is defined by 
   \[
   {\cal E}_{k-1}(x,Z_{k})  = x + \Delta b(t_{k-1},x) + \sqrt{\Delta} \sigma(t_{k-1},x) Z_{k}.
   \]
  Now, here is how we construct the algorithm. Given the distribution of $\bar X_0$, we quantize $\bar X_0$ and denote its quantization  by $\widehat X_0^{\Gamma_0}$. We want now to define the recursive  marginal  quantization of  $\bar X_{t_1}$. Owing to   Equation (\ref{EqDistAsDistXtkIntro}) and given that the previous marginal $\bar X_0$ has already be quantized,  we  replace  $\bar X_0$ by $\widehat X_0^{\Gamma_0}$, then,  we set $\widetilde X_{t_{1}}  :=  {\cal E}_0(\widehat X_{0}^{\Gamma_0},Z_{1})$  and consider the induced distortion 
  \[
   \widetilde D_{1} (\Gamma)  :=  \mathbb E \big[   {\rm dist}(\widetilde X_{t_{1}}, \Gamma)^2 \big] =  \mathbb E \big[   {\rm dist}({\cal E}_0(\widehat X_{0}^{\Gamma_0} ,Z_{1}), \Gamma)^2 \big],
  \]
  where  $\Gamma \subset \mathds R^d$ and   ${\rm card }(\Gamma) \le N_1$.  The distortion function $ \widetilde D_{1} (\cdot)$ is the one   to be optimized in order to produce the optimal $N_1$-quantizer $\Gamma_1$. Consequently, we define the recursive marginal quantization of $\bar X_{t_1}$ as the optimal quantization  $\widehat X_{t_1}^{\Gamma_1}$ of $\widetilde X_{t_1}$: $\widehat X_{t_1}^{\Gamma_1}   = {\rm Proj}_{\Gamma_1}(\widetilde X_{t_1})$, where 
 \[
 \Gamma_1 \!\in  \argmin \{ \widetilde D_{1}(\Gamma),\  \Gamma \subset \mathbb R^d, \ {\rm card}(\Gamma) \leq N_{1} \}.
 \]
 Once the  optimal $N_1$-quantizer $\Gamma_1$ is produced, we  define as previously  the recursive marginal quantization of  $\bar X_{t_2}$ as the optimal quantization $\widehat X_{t_2}^{\Gamma_2}$ of $\widetilde X_{t_1}$ where 
 \begin{eqnarray*}
 && \Gamma_2 \!\in  \argmin \{ \widetilde D_{2}(\Gamma),\  \Gamma \subset \mathbb R^d, \ {\rm card}(\Gamma) \leq N_{2} \}  \\
 &&   \widetilde D_{2}(\Gamma) =  \mathbb E \big[   {\rm dist}(\widetilde X_{t_{2}}, \Gamma)^2 \big]  \quad \textrm{ and } \quad  \widetilde X_{t_{2}}  :=  {\cal E}_1(\widehat X_{1}^{\Gamma_1},Z_{2}).
 \end{eqnarray*}
Repeating this procedure, we  define the recursive marginal quantization of $(\bar X_{t_k})_{0 \le k \le n}$ as the optimal quantizations $(\widehat X_{t_k}^{\Gamma_k})_{0 \le k \le n}$  of  the process $(\widetilde X_{t_k})_{0 \le k \le n}$: $\forall \ k \in \{0, \ldots, n \}$, $\widehat X_{t_k}^{\Gamma_k}   = {\rm Proj}_{\Gamma_k}(\widetilde X_{t_k})$, with  $\widetilde X_0 = \bar X_0$.   This leads us to consider the sequence of  recursive  marginal quantizations $(\widehat X_{t_k}^{\Gamma_k}  )_{k=0,\ldots,N}$ of  $(\bar X_{t_k}  )_{k=0,\ldots,N}$, defined from the following recursion: 
 \begin{eqnarray*}
  \widetilde X_0  &= &\bar X_0 \\
\widehat X_{t_k}^{\Gamma_k}   &= &{\rm Proj}_{\Gamma_k}(\widetilde X_{t_k})    \quad \textrm{and}\quad   \widetilde X_{t_{k+1}} = {\cal E}_k(\widehat X_{t_k}^{\Gamma_k} ,Z_{k+1}) ,\; k=0,\ldots,n-1
  \end{eqnarray*}
where $(Z_k)_{k=1,\ldots, n}$ is an  i.i.d. sequence of  ${\cal N}(0;I_q)$-distributed random vectors, independent of  $\bar X_0$.

From an analytical point of view, this approach  raises some new  challenging   problems among which the estimation of the  quadratic error bound  $\Vert \bar X_{t_k}  - \widehat X_{t_k}^{\Gamma_k}  \Vert_{_2} : = \big(\mathbb E\vert  \bar X_{t_k} - \widehat X_{t_k}^{\Gamma_k}  \vert_{_2}\big)^{1/2}$, for every $k=0, \ldots,n$.  
    We  will  show in particular that for any sequence $(\widehat X_{t_k}^{\Gamma_k})_{0 \le k \le n}$ of (quadratic) optimal quantization  of  $(\widetilde X_{t_k})_{0 \le k \le n}$, the quantization error  $\Vert \bar X_{t_{k}} -  \widehat X_{t_{k}}^{\Gamma_k}  \Vert_{_2}$, at the step $k$ of the recursion,   is bounded by the cumulative quantization errors  $\Vert \widetilde X_{t_{i}} -  \widehat X_{t_{i}}^{\Gamma_i}  \Vert_{_2}$, for $i=0, \cdots,k$.  Owing to  the non-asymptotic bound for the quantization errors $\Vert \widetilde X_{t_{i}} -  \widehat X_{t_{i}}^{\Gamma_i}  \Vert_{_2}$,  known as Pierce's Lemma (which will be recalled further on in Section  \ref{SecOptiQuant}) we precisely show 
that   for  every $k=0, \ldots, n$, for   any $\eta\!\in ]0,1]$,
 \begin{equation*} 
  \Vert  \bar X_k -  \widehat X_k^{\Gamma_k} \Vert_{_2}  \le    \sum_{\ell=0}^{k}  a_{\ell}  N_{\ell} ^{-1/d},
 \end{equation*}
where $a_{\ell}$
 is a positive real constant depending on $b$, $\sigma$, $\Delta$, $x_0$, $\eta$ (see Theorem~\ref{TheoConvergenceRate} further on).

The paper is organized as follows. We recall first some basic facts about  (regular) optimal quantization in Section~\ref{SecOptiQuant}.   The marginal quantization method is described  in Section~\ref{SectMQ}. We give in this section the induced quantization error. In section~\ref{SectionComMarQuan}, we   illustrate  how to get the optimal grids using a deterministic zero search  Newton-Raphson's algorithm  and  show how to estimate the associated weights and transition probabilities.  The last section, Section~\ref{SectNum}, is devoted to  numerical examples. We first compare the recursive marginal quantization of $W_1$ with its regular marginal quantization (see Section~\ref{SecOptiQuant}), where $(W_t)_{t\!\in [0,1]}$ stands for the Brownian motion.   Secondly, we use the marginal quantization method for   the pricing of an European  Put option in a local volatility model (as well as in the Black-Scholes model) and   compare the results  with those obtained from  the Monte Carlo method.  

\bigskip 
\noindent {\sc Notations}.  We denote by ${\cal M}(d,q,\mathbb R)$ the set of $d \times q$ real-valued matrices.  If $A =[a_{ij}]\!\in {\cal M}(d,q,\mathbb R)$,  $A^{\star}$ denotes its  transpose  and we define the norm   $\Vert  A \Vert := \sqrt{{\rm Tr}(AA^{\star})} = (\sum_{i,j} a_{ij}^2)^{1/2}$, where ${\rm Tr}(M)$ stands for the trace of $M$, for $M\!\in {\cal M}(d,d,\mathbb R)$. For every $g:\mathbb R^d \to {\cal M}(d,q,\mathbb R)$,  we will set $[g]_{\rm Lip}= \sup_{x\neq y}\frac{\Vert g(x)-g(y) \Vert}{\vert x-y \vert}$.  For $x, y\!\in \mathbb R$, $x \vee y  = \max(x,y) $. For $x\!\in \mathbb R^d$ and $E \subset  \mathbb R^d$, ${\rm dist}(x,E)=\!\inf_{\xi\!\in \mathbb R^d} |x-\xi|$ (distance of $x$ to $E$ with respect to the current norm $|\,.\, |$ on $\mathbb R^d$).

\section{Background  on optimal quantization}  \label{SecOptiQuant}
 Let $(\Omega,\mathcal{A},\mathbb{P})$ be a  probability space and  let   $X : (\Omega,\mathcal{A},\mathbb{P}) \longrightarrow  \mathbb{R}^{d} $ be a random variable with  distribution $\mathbb{P}_X$. Assume $\mathbb R^d$ is equipped with a norm $|\,.\,|$. Assume $X\!\in L^r(\mathbb P)$ $i.e.$ ${\Vert X \Vert}_r := {\left(\mathbb E \vert X \vert ^r \right)}^{1/r}<+\infty$ for a given $r\!\in (0,+\infty)$ ($\mathbb E$ denotes the expectation with respect to $\mathbb P$).
 
  The {\em $L^r$-optimal quantization  problem at level  $N$ } for   the random vector $X$ (or, in fact, its distribution $\mathbb P_X$)  consists in finding  the best approximation of  $X$   in $L^r(\mathbb P)$ by a function  $\pi(X)$ of  $X$  where $\pi :\mathbb R^d\to \mathbb R^d$ is a Borel function  taking  at most  $N$ values. 
%
 In formalized terms, this amounts out to solve  the  minimization problem: 
$$ 
e_{N,r}(X) =\!\inf{\big\{ \Vert X - \pi(X) \Vert_{r}, \pi: \mathbb{R}^d \rightarrow  \Gamma, \Gamma \subset \mathbb{R}^d,  \vert \Gamma \vert \leq N \big\}},
$$
where $\vert A \vert $ stands for the cardinality of a subset  $A \subset \mathbb R^d$.  Note that $e_{N,r}(X) $ only depends on the distribution $\mathbb P_{X}$ of $X$.

Any  $\Gamma = \{x_1,\ldots,x_N  \}\subset \mathbb{R}^d$ of size $N$  and   any Borel function $\pi :\mathbb R^d\to  \Gamma$ are both often called  {\em $N$-quantizer}. The terminology   {\em grid}  (or a codebook in coding theory) of size $N$  is more specifically used for $\Gamma$ itself. The resulting  random vector $\pi(X)$ is called an {\em $N$-quantization} of $X$. 

Such a grid $\Gamma$ induces   Voronoi diagrams or partitions  $(C_i(\Gamma))_{ i=1,\ldots,N}$ of $\mathbb{R}^d$ defined as Borel partitions of  $\mathbb R^d$  satisfying 
$$ 
\forall\, i\!\in \{1,\cdots,N\},\quad  C_i(\Gamma) \subset  \big\{ x\!\in \mathbb{R}^d : \vert x-x_i \vert = \min_{j=1,\ldots,N}\vert x-x_j \vert \big\}.
$$ 
To such a Voronoi partition is attached a nearest neighbor projection $\pi_{\Gamma} =\sum_{i=1}^N x_i \mbox{\bf{1}}_{C_i(\Gamma)}$ and a {\em Voronoi $\Gamma$-valued quantization} of $X$ denoted $ \widehat{X}^{\Gamma} = \pi_{\Gamma}(X)$ reading:
$$
 \widehat{X}^{\Gamma} = \sum_{i=1}^N  x_i  \mbox{\bf{1}}_{\{X\!\in C_i(\Gamma)\}} . 
$$ 
Then, for any Borel function $\pi: \mathbb{R}^d \rightarrow \Gamma$ and any (Borel) nearest neighbor   projection $\pi_{\Gamma}$, 
$$ 
\vert X -\pi(X) \vert \geq \min_{i=1,\ldots,N} {\rm dist}(X,x_i) = {\rm dist}(X,\Gamma)=\vert X -\pi_{\Gamma}(X) \vert=  \vert X - \widehat{X}^{\Gamma} \vert, 
$$  
so that the optimal  $L^r$-mean quantization error  $e_{N,r}(X) $ reads
\begin{equation}
 e_{N,r}(X)  = \inf\big\{ \Vert X - \widehat{X}^{\Gamma} \Vert_r, \Gamma \subset \mathbb{R}^d,  \vert \Gamma \vert  \leq N \big\} 
= \!\inf_{ \substack{\Gamma  \subset \mathbb{R}^d \\ \vert \Gamma \vert  \leq N}} \left(\int_{\mathbb{R}^d} {\rm dist}(\xi,\Gamma)^r \mathbb P_X(d\xi) \right)^{1/r}. \label{er.quant}
 \end{equation}
 Let us recall the following basic facts:
 
\smallskip
-- If $|\,.\,|$ is an Euclidean norm, the boundary of the Voronoi are contained in finitely many hyperplanes so that, if the distribution of $X$ assigns no mass to hyperplanes, two $\Gamma$-valued Voronoi quantizations of $X$ are $\mathbb P$-$a.s.$ equal  since they only differ on the boundaries of the Voronoi diagram.

\smallskip
-- for every  $N \geq 1$, the infimum in~\eqref{er.quant} is a minimum since it is attained  at  least  at one  quantization grid $\Gamma^{N, \star}$ of size at most $N$. Any  resulting  $\Gamma^{N, \star}$  is called an  {\em $L^r$-optimal  quantizer at level $N$}. 

-- If    $ \vert {\rm supp}(\mathbb P_X))  \vert \geq N$ then  any $L^r$-mean  optimal  quantizer at level $N$  has exactly full size   $N$ (see~\cite{GraLus} or~\cite{Pag}). 

\smallskip
-- The    $L^r$-mean quantization error $e_{N,r}(X)$ decreases to zero as the level  $N$ goes to infinity and its sharp rate of convergence is ruled by the so-called Zador Theorem.  There  is also   a non-asymptotic universal upper bound known as   Pierce's Lemma. Both  results hold for any norm on $\mathds R^d$ and are  recalled   below. This second result will allow us  to put a finishing  touch to the  proof of our  main theoretical result, stated in  Theorem~\ref{TheoConvergenceRate}.

 \begin{thm} \label{thm:Zador} $(a)$ {\bf Zador Theorem}, (see~\cite{GraLus}).  Let $r>0$. Let $X$ be an $\mathbb R^d$-valued random  vector  such that  $\mathbb{E}\vert X \vert^{r+ \eta} < +\!\infty \textrm{ for some } \eta >0 $ and  let $\mathbb P_X= f \cdot \lambda_d + P_s$ be the   decomposition of $\mathbb P_X$ with respect to the Lebesgue measure $\lambda_d$ where  $P_s$ denotes its singular   part. Then 
 \begin{equation}
  \lim_{N \rightarrow +\infty} N^{\frac{1}{d}} e_{N,r}(X) = \widetilde Q_r(\mathbb P_X)
  \end{equation}
where 
$$ 
\widetilde Q_r(\mathbb P_X) = \widetilde J_{r,d} \left(\!\int_{\mathbb{R}^d} f^{\frac{d}{d+r}} d\lambda_d \right)^{\frac{1}{r} +\frac{1}{d}} = \widetilde J_{r,d} \ \Vert f \Vert_{\frac{d}{d+r}}^{1/r} \!\in [0,+\infty), 
$$
 $$ 
 \widetilde J_{r,d} =\!\inf_{N \geq 1} N^{ \frac{1}{d}} e_{N,r} \big(U([0,1]^d)\big)\!\in (0,+\infty)
$$
and  $ U([0,1]^d) $ denotes the uniform distribution over the hypercube $[0,1]^d$.  

\medskip
\noindent $(b)$ {\bf Pierce Lemma} (see~\cite{GraLus, LusPag}).   Let $r,\, \eta>0$. There exists  a universal constant  $K_{r,d,\eta}$  such that  for every   random vector $X:(\Omega,{\cal A}, \mathbb P) \rightarrow \mathbb R^d$,  
 \begin{equation}  \label{LemPierce} 
\inf_{\vert  \Gamma  \vert  \leq N} \Vert   X   - \widehat X^{\Gamma}  \Vert_{_r}   \le  K_{r, d,\eta} \, \sigma_{r+\eta}(X) N^{-\frac{1}{d}}, 
 \end{equation} 
 where $ \sigma_{p}(X)$ is the $L^p$-pseudo-standard deviation of $X$ defined by 
 $$  
 \sigma_{p}(X) =\!\inf_{\zeta\in \mathbb R^d} \Vert X-\zeta  \Vert_{_{p}} \le +\infty.
 $$
\end{thm}

For more details on the universal contant $K_{2,d,\eta}$, we refer to~\cite{LusPag}. We will call $\widetilde Q_r(\mathbb P_X)$ the Zador's constant associated to $X$.  

From a Numerical Probability viewpoint, finding an optimal $N$-quantizer  $\Gamma$ may be a challenging task even in  the quadratic canonical Euclidean case   $i.e.$ when $r=2$ and $|\,.\,|$ is the canonical Euclidean norm on $\mathbb R^d$. This framework is  the case of interest for    numerical applications. The starting point for numerical applications is to note that  if we have access to  a ``good" quantization   $\widehat X^{\Gamma}$  close enough to $X$ in distribution, then,  for every continuous  function $f: \mathbb R^d \rightarrow  \mathbb R$, we can approximate $\mathbb E f(X)$ (when finite) by the cubature formula:  
\begin{equation} \label{QuantProcedureEstim}
\mathbb{E}  f \big(\widehat{X}^{\Gamma} \big) = \sum_{i =1}^N  p_i f(x_i)
\end{equation}
where $p_i =\mathbb{P}( \widehat{X}^{\Gamma} = x_i )$, $i=1,\ldots,N$.  By {\em access}  we mean to have numerical values for the grid $\Gamma$ and its {\em weights} $p_i$, $i=1, \ldots,N$, or equivalently to the distribution of $\widehat X^{\Gamma}$.

Among all good quantizers, the so-called {\em stationary} quantizers defined below play an important role because they provide higher order cubature formulas (see below) and are also the only class of quantizers that can be reasonably computed owing to its connection with the  critical point  of the quadratic distortion also defined below.   

\begin{defn} A quantizer  $\Gamma = \{ x_1,\ldots,x_N \}$  of size $N$ inducing the Voronoi quantization $\widehat{X}^{\Gamma}$  of $X$  is  stationary  if  
\begin{eqnarray}
(i)&& \hskip -0,5cm   \mathbb{P}\left(X\!\in \cup_{i} \partial  C_i(\Gamma) \right) =0 \hskip 8cm 
\nonumber\\
(ii)&&  \hskip -0,5cm \mathbb{E} \big(X  \vert  \widehat{X}^{\Gamma} \big) =\widehat{X}^{\Gamma} \; \mathbb P\mbox{-}a.s.\label{EqStation}
\end{eqnarray}
 \end{defn}
 Equation $(ii)$ can be re-written as 
 \[
 x_i = \mathbb E \big(X\,|\,X\!\in C_i(\Gamma) \big) = \frac{\mathbb E(X\mbox{\bf 1}_{\{X\!\in C_i(\Gamma)\}})}{\mathbb P(X\!\in C_i(\Gamma))},\; i=1,\ldots,N,
 \]
 with the convention that  the equality is always true when $\mathbb P(X\!\in C_i(\Gamma))=0$.
 
 This notion of stationarity is closely related to the critical point of the so-called {\em quadratic distortion} function defined on $(\mathbb R^d)^N$  by  
\begin{equation}\label{EqDistor}
 D_{N,2} (x)  =  \mathbb E \big(\min_{1\le i\le N}|X-x_i|^2\big)= \int_{\mathbb R^d}|\xi-x_i|^2\mathbb P_{X}(d\xi), \quad x=(x_1,\ldots,x_{_N})\!\in (\mathbb R^d)^N.
  \end{equation}
  
This function is clearly symmetric.  Its connection with the quadratic mean quantization errors is as follows: set $\Gamma= \{x_i, \; i=1,\ldots,N\}$ whose cardinality is at most $N$. Then, with obvious notations, 
  \[
  D_{N,2} (x)= \mathbb E\big({\rm dist}(X, \Gamma)^2\big) = \int_{\mathbb{R}^d} {\rm dist}(\xi,\Gamma)^2 \mathbb P_X(d\xi)=\sum_{a\in \Gamma} \int_{C_a(\Gamma)} \vert \xi - a\vert^2 \mathbb P_X(d\xi).
  \]

As any grid of size at most $N$ can be ``represented'' by some $N$-tuples (by repeating, if necessary, some of its  elements), we deduce straightforwardly that 
 \[
 e_{N,2}^2(X) =\!\inf _{(x_1,\ldots,x_N)\in (\mathbb R^d)^N} D_{N,2}(x_1,\ldots,x_N).
 \]
 
 The function $D_{N,2}$ is continuous but also differentiable at any $N$-tuple having pairwise distinct components with a $\mathbb P$-negigible Voronoi partition boundary. 
The following proposition makes this more precise. 
 \begin{prop} \label{PropDifferentiability}(see~\cite{GraLus, Pag})
 $(a)$ The function $D_{N,2}$ is differentiable at any $N$-tuple $(x_1,\ldots,x_{_N}) \!\in (\mathbb R^{d})^N$ having pairwise distinct components   and such that $\mathbb{P}\left(X\!\in \cup_{i} \partial  C_i(\Gamma) \right)=0$. Furthermore, we have
 \begin{eqnarray}
  \nabla D_{N,2}(x_1,\ldots,x_{_N})  &= & 2 \Big(   \!\int_{C_i(\Gamma ) }  (x_i - x ) d \mathbb P_X (x)  \Big)_{i=1,\ldots,N}\\
  &=& 2 \Big(\mathbb P(X\!\in C_i(\Gamma))x_i-\mathbb E(X\mbox{\bf 1}_{\{X\!\in C_i(\Gamma)\}}) \Big)_{i=1,\ldots,N}.
 \end{eqnarray}
  
  \smallskip
  \noindent $(b)$ A grid $\Gamma= \{x_1,\ldots,x_{_N}\}$ of full size $N$ is stationary if and only if 
  \begin{equation}\label{eq11}
  \mathbb{P}\left(X\!\in \cup_{i} \partial  C_i(\Gamma) \right) =0 \quad \mbox{ and }\quad    \nabla D_{N,2} (\Gamma) = 0.
\end{equation}

\smallskip 
  \noindent   $(c)$  If the support of $\mathbb P_{_X}$ has at least $N$ elements, any $L^2$-optimal quantizer at level $N$ has full size and a $\mathbb P$-negligible Voronoi boundary. Hence it is a  stationary $N$-quantizer.  
\end{prop}
 
 \noindent {\sc Convention:} To alleviate notations, we will often put grids of all size $N$ as an argument of the distortion function $D_{2,N}$ as well as for its gradient when its Voronoi boundary is negligible.  This has already been done implicitly in the introduction. 
 
 \medskip

When approximating  $\mathbb E f(X)$ by $\mathbb E f(\widehat X^{\Gamma})$ and $f$ is smooth enough (say ${\cal C}^{1+\alpha}$) and  $\Gamma$ is a  stationary $N$-quantizer, the resulting  error may be bounded by the quantization error $\| X - \widehat X^{\Gamma} \|_2$. We briefly  recall some error bounds induced from the approximation of  $\mathbb Ef(X)$ by  (\ref{QuantProcedureEstim}) (we refer to~\cite{PagPri} for further details).  
\begin{itemize}
\item[(a)]  Let $\Gamma$ be a stationary quantizer    and let  $f$ be a Borel function on $\mathbb R^d$.   If $f$ is  a convex function then
\begin{equation}  \label{EqIneqConv}
\mathbb{E} f(  \widehat{X}^{\Gamma}) \leq \mathbb{E} f(X).
\end{equation}

\item[(b)]  Lipschitz  continuous functions: 
\begin{itemize}

\item  If $f$ is Lipschitz continuous then  for any  $N$-quantizer $\Gamma$ we have
$$
\big \vert \mathbb{E}  f(X) - \mathbb{E}  f( \widehat{X}^{\Gamma})   \big \vert \leq [f]_{{\rm Lip}} \Vert X-   \widehat{X}^{\Gamma} \Vert_{_2}.
$$
\item Let $\theta: \mathbb R^d\rightarrow \mathbb{R}_{+}$ be a nonnegative convex function such that $\theta(X)\!\in L^{2}(\mathbb{P})$. If $f$ is locally Lipschitz with at most $\theta$-growth, i.e. $\vert  f(x)-f(y) \vert \leq [f]_{{\rm Lip}} \vert x-y \vert (\theta(x)+\theta(y))$ then $f(X)\in L^{1}(\mathbb{P})$ and
$$
\big \vert \mathbb{E}  f(X) - \mathbb{E} f( \widehat{X}^{\Gamma})  \big \vert \leq 2 [f]_{{\rm Lip}} \Vert X- \widehat{X}^{\Gamma} \Vert_{_2} \Vert \theta(X) \Vert_{_2}.
$$
\end{itemize}

\item[(c)] Differentiable functions:  if $f$ is differentiable on $\mathbb R^d$ with an $\alpha$-H\"older gradient  $\nabla f$ ($\alpha\!\in [0,1]$), then for any stationary $N$-quantizer $\Gamma$,
$$
\big \vert \mathbb{E}  f(X)  - \mathbb{E} f(\widehat{X}^{\Gamma} ) \big \vert \leq [\nabla  f]_{\alpha, {\rm Hol}} \Vert X-  \widehat{X}^{\Gamma} \Vert_{_2}^{1+\alpha}.
$$
\end{itemize}


\section{Recursive marginal quantization of the Euler process}  \label{SectMQ}
Let $(X_t)_{t \geq 0}$ be a stochastic process taking values in a  $d$-dimensional Euclidean space $\mathbb R^d$ and  solution to  the stochastic differential equation: 
\begin{equation} \label{EqSignalProcess}
 dX_t = b(t,X_t) dt + \sigma(t,X_t) dW_t,   \quad   X_0  \!\in \mathbb R^d, 
 \end{equation}
where  $W$ is a  standard  $q$-dimensional  Brownian motion starting  at $0$ and where   $b:[0,T] \times \mathbb R^d \rightarrow  \mathbb R^d$ and the matrix diffusion coefficient function  $\sigma:[0,T] \times \mathbb R^d \rightarrow  {\cal M}(d,q, \mathbb R) $ are measurable and satisfy 
\begin{equation}\label{BoundEDS}
b(.,0) \; \mbox{ and } \;\sigma(.,0) \;\mbox{ are bounded on } \; [0,T]
\end{equation}
and the uniform  global Lipschitz continuity assumption : for every $t\!\in [0,T]$ and every $x\!\in  \mathbb R^d$,
\begin{equation}   \label{LipEDS} 
 \vert b(t,x) - b(t,y) \vert    \le  [b]_{\rm Lip} \vert  x-y  \vert  \quad \mbox{ and }\quad  \Vert  \sigma(t,x) - \sigma(t,y)\Vert  \leq  [\sigma]_{\rm Lip} \vert  x-y  \vert 
.
\end{equation}  
In particular, this implies that 
 \begin{equation}  \label{EqLinGrowth}
\vert b(t,x)   \vert  \leq L (1 +\vert x \vert)   \  \textrm{  and  } \  \Vert \sigma(t,x)   \Vert  \leq  L  (1 +\vert x \vert) 
 \end{equation} 
with  $L= \max\big( [b]_{\rm Lip}, \|b(.,0)\|_{\sup},  [\sigma]_{\rm Lip}, \|\sigma(.,0)\|_{\sup}\big)$. These  assumptions  guarantee the existence of  a unique strong solution of (\ref{EqSignalProcess}) starting from any $x_0\!\in \mathbb R^d$. 

\subsection{The algorithm}

Consider the Euler scheme of the process $(X_t)_{t \geq 0}$ starting from $\bar X_0= X_0$:
$$
\bar X_{t_{k+1}} = \bar X_{t_k} + \Delta  b(t_k,\bar X_{t_k})  + \sigma(t_k,\bar X_{t_k}) (W_{t_{k+1}} - W_{t_k}),
$$
where $t_k =\frac{k T}{n}=k\Delta$ for every $k\!\in \{ 0, \ldots,n \}$. 
%

\bigskip 
\noindent {\sc Notation simplification}.  To alleviate   notations, we  set 
\begin{eqnarray*}
 & &  Y_k:= Y_{t_k}   \textrm{   (for any process $Y$ evaluated at time } t_k ) \\
 &  & b_k(x):= b(t_k,x) , \quad x\!\in \mathbb R^d\\
 & &  \sigma_k(x) := \sigma(t_k,x), \quad x\!\in \mathbb R^d. 
\end{eqnarray*}

 Recall that the distortion function  $\bar D_{k}$ associated to $\bar  X_{k}$  may be written   for every $k=0, \ldots,n-1$, as
   \[
   \bar  D_{k+1}(\Gamma_{k+1}) =  \mathbb E \big[   {\rm dist}({\cal E}_k(\bar X_k,Z_{k+1}), \Gamma_{k+1})^2 \big]  \label {EqDistorNonQuant}
  \]
  where 
  \[
  {\cal E}_k(x,z) := x+\Delta b(t_k,x)+\sqrt{\Delta} \sigma(t_k,x) z,\;  x\!\in \mathbb R^d, \; z \!\in \mathbb R^q.
  \]
Supposing that  $\bar  X_0$ has already been quantized  by  $\widehat X_0^{\Gamma_0}$ and setting $\widetilde X_{1}  =  {\cal E}_0(\widehat X_0^{\Gamma_0} ,Z_{1})$, we may approximate the distortion function $\bar D_{1}(\Gamma_1)$ by
     \begin{eqnarray*}
  \widetilde D_{1} (\Gamma_{1}) & := &   \mathbb E \big[   {\rm dist}(\widetilde X_{1}, \Gamma_{1})^2 \big] \\
  & = & \mathbb E \big[   {\rm dist}({\cal E}_0(\widehat X_0^{\Gamma_0} ,Z_{1}), \Gamma_{1})^2 \big] \\ 
  &  = &  \sum_{i=1}^{N_{0}}  \mathbb E \big[   {\rm dist}({\cal E}_0(x_i^{0},Z_{1}), \Gamma_{1})^2 \big] \mathbb P\big(\widehat X_0^{\Gamma_0} =x_i^{0} \big). 
 \end{eqnarray*}
 
 This allows us (as already said in the introduction) to consider the sequence of recursive (marginal) quantizations $(\widehat X_k^{\Gamma_k})_{k=0,\ldots,n}$ defined from the following recursion: 
 \begin{equation}\label{eq:hattilde-k}
\widehat X_k ^{\Gamma_k} = {\rm Proj}_{\Gamma_k}(\widetilde X_k)  \quad \textrm{and} \quad   \widetilde X_{k}  =  {\cal E}_k(\widehat X_{k-1}^{\Gamma_k},Z_{k}),\; k=1,\ldots,n,\;    \widetilde X_0  = \bar X_0.
 \end{equation}
where $(Z_k)_{k=1,\ldots, n}$ is an   i.i.d., sequence of ${\cal N}(0;I_q)$-distributed random vectors, independent of  $\bar X_0$.

 \subsection{The error analysis}
 
Our aim is now to compute the quantization error bound  $\Vert  \bar X_T - \widehat X_T \Vert_{_2} : = \Vert  \bar X_n - \widehat X_n^{\Gamma_n} \Vert_{_2}$.  The analysis of this error bound  will  be the subject of the following  theorem,  which is the main result of the paper.   In this section, we assume for convenience that 
$\widetilde X_0=X_0=x_0 $ (which amounts to conditioning with respect to $\sigma(X_0)$ since $W$ and $X_0$ are independent).
 
 \begin{thm}  \label{TheoConvergenceRate}  Let the coefficients $b$, $\sigma$    satisfy  the assumptions~\eqref{BoundEDS} and~\eqref{LipEDS}.
 Let for every $k=0,\ldots,n$,   $\Gamma_k$ be a quadratic optimal quantizer for  $\widetilde X_k$ at level  $N_k$. Then,  for  every $k=0, \ldots, n$, for   any $\eta\!\in (0,1]$,
 \begin{equation}  \label{EqTheoConvergenceRate}
  \Vert  \bar X_k -  \widehat X_k^{\Gamma_k} \Vert_{_2}  \le  K_{2,d,\eta}  \sum_{\ell=1}^{k}  a_{\ell}(b, \sigma,t_k, \Delta,x_0 ,L,2+\eta)   N_{\ell} ^{-1/d}
 \end{equation}
  where $K_{2,d,\eta}$ is a universal constant  defined in  Equation~\eqref{LemPierce} and, for every $p\!\in (2,3]$,
  \[
 a_{\ell}(b, \sigma,t_k, \Delta,x_0 ,L,p) :=  e^{C_{b,\sigma}  \frac{(t_k-t_{\ell})}{p} }  \Big[  e^{(\kappa_p  + K_p) t_{\ell}}   \vert  x_{0} \vert^p   +   \frac{e^{\kappa_p  \Delta }L + K_p}{\kappa_p +K_p}  \big(e^{(\kappa_p + K_p) t_{\ell}} -1 \big) \Big]^{\frac{1}{p}}
  \]
with    $\,C_{b,\sigma}=[b]_{\rm Lip} + \frac{1}{2} [\sigma]_{\rm Lip}^2$,  
 $ \kappa_p := \frac{(p-1)(p-2)}{2 } + 2 p L $ and  $\displaystyle K_p := 2^{p-1}L{}^p \Big( 1 + p + \Delta^{\frac{p}{2}-1} \Big)   \mathbb E \vert  Z \vert^p$,  $Z \sim {\cal N}(0;I_q)$.
 \end{thm}
 
Let us make the following remarks.
 
\begin{rem}  
 \noindent  It is crucial  for applications   to notice that the real constants $a_{\ell}(\cdot,t_k,\cdot,\cdot,\cdot,\cdot,p)$ do not explode when $n$ goes to  infinity and we have  
 \[
 \sup_{n  \ge 1} \, \max_{0 \le \ell \le k \le n}  a_{\ell}(\cdot,t_k,\cdot,\cdot,\cdot,\cdot,p)  \le e^{C_{b,\sigma} \frac{T}{p} }  \Big[  e^{(\kappa_p  + K_p^{\star}) T}   \vert  x_{0} \vert^p   +   \frac{e^{\kappa_p  T } L + K_p^{\star}}{\kappa_p }  \big(e^{(\kappa_p + K_p^{\star}) T} -1 \big) \Big]^{\frac{1}{p}},
 \] 
where  $K_p^{\star} =  2^{p-1}L{}^p \Big( 1 + p + T^{\frac{p}{2}-1} \Big)$.   We also remark  that  $\kappa_p \le 2(1+p)L$.  
 \end{rem}
 
 The proof of the theorem relies on the Lemma   below, which  proof is  postponed to the appendix.
 
 \begin{lem}    \label{PropPrinc} Let the coefficients $b$, $\sigma$ of the diffusion satisfy the assumptions~\eqref{BoundEDS} and~\eqref{LipEDS}. Then, for every $p\!\in (2,3]$, for every $k=0, \ldots,n$,
  \begin{equation}  \label{EqPropPrinc}
  \mathbb E\vert \widetilde X_{k}\vert^{p} \le   e^{(\kappa_p  + K_p) t_k}   \vert  x_{0} \vert^p   +      \frac{e^{\kappa_p  \Delta } L + K_p}{\kappa_p + K_p } \big( e^{(\kappa_p + K_p) t_k}-1\big),
  \end{equation} 
  where $K_p$ and $\kappa_p$ are  defined in Theorem~\ref{TheoConvergenceRate}.
 \end{lem}
 
 Let us prove the theorem.
\begin{proof}[{\bf Proof } (of Theorem~\ref{TheoConvergenceRate})]
First we note that for every $k=0,\ldots,n$,
\begin{equation} \label{EqDecompError}
\Vert  \bar X_k - \widehat X_k  \Vert_{_2}  \le \Vert   \bar  X_k - \widetilde X_k  \Vert_{_2}  + \Vert  \widetilde X_k - \widehat X_k^{\Gamma_k}  \Vert_{_2} .
\end{equation}
Let us control the first term of the right hand side of the above equation.  To this end, we first note that, for every $k=0, \ldots,n$,   the function ${\cal E}_k(\cdot,Z_{k+1})$ is Lipschitz  w.r.t. the  $L^2$-norm:  in fact, for every $x, x'\!\in \mathbb R^d$,
   \begin{eqnarray*}   
  \mathds E \vert {\cal E}_k(x, Z_{k+1})  -  {\cal E}_k(x',Z_{k+1})  \vert ^2 &\le&  \big(1+ \Delta \big(2[b_k]_{\rm Lip} +[\sigma_k]^2_{\rm Lip} \big)+ \Delta^2  [b_k(.)]_{\rm Lip}^2  \big) \vert x-x' \vert^2 \\
   &\le& \big(1+ \Delta \big(2[b]_{\rm Lip} +[\sigma]^2_{\rm Lip} \big)+ \Delta^2  [b]_{\rm Lip}^2  \big) \vert x-x' \vert ^2\\
   &\le & (1+\Delta C_{b,\sigma})^2 \vert x-x' \vert^2\\
   &\le& e^{ 2 \Delta C_{b,\sigma}} \vert x-x' \vert^2,
    \end{eqnarray*} 
where  $C_{b,\sigma} = [b]_{\rm Lip} + \frac{1}{2} [\sigma]_{\rm Lip}^2$ does not depend on $n$. Then, it follows  that  for every $\ell =0, \ldots,k-1$,
\begin{eqnarray}  \label{EqUse}
  \Vert \bar X_{\ell+1} - \widetilde X_{\ell+1} \Vert_{_2} & = &   \Vert {\cal E}_{\ell}(\bar X_{\ell}, Z_{\ell+1})  -  {\cal E}_{\ell}(\widehat X_{\ell}^ {\Gamma_{\ell}},Z_{\ell+1})  \Vert_{_2} \nonumber \\
   &  \le &    e^{  \Delta C_{b,\sigma}    }   \Vert \bar X_{\ell} - \widehat X_{\ell}^{\Gamma_{\ell}} \Vert_{_2} \nonumber \\
 & \le &  e^{   \Delta C_{b,\sigma}    }   \Vert  \bar X_{\ell} -  \widetilde X_{\ell} \Vert_{_2} +    e^{   \Delta C_{b,\sigma}    }   \Vert  \widetilde X_{\ell} -  \widehat X_{\ell}^{\Gamma_{\ell}}   \Vert_{_2}. 
 \end{eqnarray}
Then, we show by a backward induction using~\eqref{EqDecompError} and~\eqref{EqUse}  that    (recall that  $\widetilde X_0 = \widehat X_0^{\Gamma_0} = x_0$)
\[  
\Vert  \bar X_k -  \widetilde X_k \Vert_{_2} \le  \sum_{\ell=1}^{k}  e^{ (k-\ell) \Delta  C_{b,\sigma} }  \Vert  \widetilde X_{\ell} -  \widehat X_{\ell}^{\Gamma_{\ell}} \Vert_{_2}.
\]
Now, we deduce from Pierce's Lemma (Equation~\ref{LemPierce} in Theorem~\ref{thm:Zador}$(b)$) and     Lemma~\ref{PropPrinc}  that,   for every $k=0, \ldots, n$,  for any $\eta \in ]0,1]$,   
\begin{eqnarray*}
 \Vert  \bar X_k -  \widehat X_k \Vert_{_2} & \le &  K_{2,d,\eta}  \sum_{\ell=1}^{k}  e^{ (k-\ell) \Delta  C_{b,\sigma} }    \sigma_{2,\eta} (\widetilde X_{\ell}) N_{\ell} ^{-1/d}  \\
 & \le  & K_{2,d,\eta}   \sum_{\ell=1}^{k}  e^{ (k-\ell) \Delta  C_{b,\sigma} }    \Vert \widetilde X_{\ell}  \Vert_{_{2+\eta}}  N_{\ell}^{-1/d} \\ 
 & \le & K_{2,d,\eta}  \sum_{\ell=1}^{k}  a_{\ell}(b, \sigma,t_k, \Delta,x_0 ,L,2+\eta)     N_{\ell} ^{-1/d},
  \end{eqnarray*}
 which is the announced result.  
\end{proof}

\medskip
\noindent {\sc Practitioner's corner}. \label{RemChoiceGridSizes}
$(a)$ {\em Optimal dispatching (to minimize the error bound).}   When we consider the upper bound of  Equation~\eqref{EqTheoConvergenceRate}, a natural question   is to  determine  how to  dispatch optimally the sizes $N_1, \cdots,N_n$ (for a fixed mesh of  length  $n$)  of the quantization grids  when we wish to use a total ``budget'' $N=N_1+ \cdots + N_n$  of elementary quantizers (with $N_k \geq 1$, for every $k=1, \ldots,n$).  Keep in mind that since $X_0$ is not random, $N_0=1$. The dispatching  problem    amounts to solving the minimization problem 
\[
\min_{N_1+\cdots+N_n = N}    \sum_{\ell=1}^{n}  a_{\ell}  N_{\ell} ^{-1/d}
\] 
where $a_{\ell}= a_{\ell}(b, \sigma,t_n, \Delta,x_0 ,L,2+\eta)$. This leads (see e.g.~\cite{BalPag}) to the  following optimal dispatching:   for every $\ell =1, \ldots, n$,
\[
N_{\ell} = \left\lfloor \frac{a_{\ell}^{ \frac{d}{d+1}}}{\sum_{k=1}^n a_{k}^{ \frac{d}{d+1}} } N \right \rfloor\vee 1,
\] 
so that  Equation~\eqref{EqTheoConvergenceRate} becomes  at the terminal instant $n$:
\begin{equation}  \label{Upper1}
  \Vert  \bar X_n -  \widehat X_n^{\Gamma_n} \Vert_{_2}  \lesssim  K_{2,d,\eta} N^{-1/d} \Big( \sum_{\ell=1}^n a_{\ell}^{ \frac{d}{d+1}}  \Big)^{1+1/d}.
\end{equation}


\noindent $(b)$  {\em Uniform dispatching (complexity minimization).} Notice that the complexity  of  the quantization tree $(\Gamma_k)_{k=0,\ldots,N}$ for the recursive marginal quantization is of order $\sum_{k=0}^{n-1} N_k N_{k+1}$. Now,   assuming  that $N_0=1$ and that for every  $N_k \leq N_{k+1}$, $k=0,\ldots,n-1$,  we want to solve (heuristically)  the problem
\begin{equation}  \label{EqMinOrig}
\min \Big\{\sum_{k=0}^{n-1} N_k N_{k+1}  \ \textrm{ subject to } \ \ \sum_{k=1}^{n} N_k =N  \Big\}.
\end{equation}
As  $\sum_{k=0}^{n-1} N_k^2  \le \sum_{k=0}^{n-1} N_k N_{k+1} \le  \sum_{k=1}^{n} N_k^2$ 
and  $N_0 =1$, this suggests that  $\sum_{k=0}^{n-1} N_k N_{k+1} \approx  \sum_{k=1}^{n} N_k^2$.  Then, if we switch to 
\begin{equation}  \label{EqMinHeur}
\min \Big\{\sum_{k=1}^{n} N_k^2 /N^2 \ \textrm{ subject to } \ \ \sum_{k=1}^{n} N_k =N  \Big\} = \min \Big\{\sum_{k=1}^{n} q_k ^2 \ \textrm{ subject to } \ \ \sum_{k=1}^{n} q_k =1  \Big\},
\end{equation}
where $q_k = N_k/N$, it is well known that the solution of this problem is given by  $q_k  = 1/n$, $i.e.$, $N_k = N/n$, for every $k=1, \ldots,n$. Plugging  this
in ~\eqref{EqMinOrig},  leads to a sub-optimal, but nearly optimal,  complexity equal to $N^2/n$. In fact, any other choice  leads to the global complexity
\[
N^2\sum_{k=0}^{n-1} q_k q_{k+1}   \ge N^2    \sum_{k=0}^{n-1} q_k^2  \ge N^2 \min \Big\{\sum_{k=0}^{n} q_k ^2   \ \textrm{ subject to }  \ \sum_{k=0}^{n} q_k =1  \Big\}    > \frac{N^2}{n}
\]
(provided $q_0$ is left free).

\smallskip
\noindent $(c)$ {\em Comparison.} If we  consider the  uniform dispatching  $ N_{\ell}  = \bar N := N/n$, $\ell=1,\ldots,n$, the error bound in Theorem~\ref{TheoConvergenceRate} becomes, still at at the terminal instant,   
\begin{equation}\label{Upper2}
  \Vert  \bar X_n -  \widehat X_n^{\Gamma_n} \Vert_{_2}  \lesssim  K_{2,d,\eta}  \Big(\frac{n}{N} \Big)^{1/d} \sum_{\ell=1}^n a_{\ell}.
\end{equation}
It is clear that the upper bound~\eqref{Upper1} is a sharper bound than~\eqref{Upper2} since,  by H\"older's Inequality, 
$$
\Big( \sum_{\ell=1}^n   a_{\ell}^{ \frac{d}{d+1}}  \Big)^{1+1/d} <   n^{1/d} \sum_{\ell=1}^n a_{\ell}.
$$
However, the uniform dispatching has smaller  complexity  than the optimal one.

\section{Computation of the marginal quantizers}  \label{SectionComMarQuan}

We focus now on the numerical   computation of  the quadratic optimal quantizers of the marginal random variable $\widetilde X_{t_{k+1}}$ given the probability  distribution function of $\widetilde X_{t_k}$. Such a task  requires the use of some algorithms like the CLVQ  algorithm,  the randomized Lloyd's algorithms  (both  requiring the computation of the gradient of the distortion function) or  Newton-Raphson's algorithm (especially for the one-dimensional setting) which involves  the gradient and the Hessian matrix of the distortion (we refer to~\cite{PagPri} for more details).      To ensure the existence of  densities (or conditional densities) of  the $\widetilde X_{t_{k+1}}$'s (given $\widetilde X_{t_k}$), we suppose that  the matrix $\sigma \sigma^{\star}(t,x)$ is invertible for every $(t,x)  \in [0,T] {\small \times} \mathds R^d$.

  
  For every $k\!\in\{0, \ldots,n\}$, let $\widehat X_k^{\Gamma}$ be  the quantization of $\widetilde X_k$ induced by   a grid $\Gamma\subset \mathbb R^d$. Recall that  the distortion function at level $N_k\!\in \mathbb N$ attached to $\widetilde X_k$  is defined on $(\mathbb R^d)^{N_k}$ by
  $$
  \widetilde D_{k}(x) = \mathbb E  \min_{1\le i\le N_k} \vert \widetilde X_k -x_i\vert ^2 = \int\min_{1\le i\le N_k} |\xi-x_i\vert^2\mathbb P_{\widetilde X_k}(d\xi),\quad  x=(x_1,\ldots,x_{_{N_k}})\!\in (\mathbb R^d)^{N_k}
  $$
 having  in mind that, if  $\Gamma_x= \{x_i, i=1,\ldots,N_k\}$,  then $  \widetilde D_{k}(x)= \Vert  \widetilde X_{k}  - \widehat X_k ^{\Gamma_x} \Vert^2_{_2}$.

  Our aim is to compute the  (at least locally) optimal quadratic  quantization grids  $(\Gamma_k)_{k=0, \ldots,n}$ associated with the random vectors $\widetilde X_k$, $k=0,\ldots,n$. Such a sequence of grids   is defined for every $k=0,\ldots,n$ by 
 \begin{equation} \label{EqOptQuant}
\Gamma_k \!\in  \argmin\big\{ \|\widehat X_k^{\Gamma}-\widetilde X_k\|_2, \Gamma \subset \mathbb R^d, {\rm card}(\Gamma) \leq N_{k}\big \}=  \argmin \big\{ \widetilde D_{k}(x),\, x\!\in  (\mathbb R^d)^{N_k}\big \}
\end{equation}
where the sequence of  (marginal) quantizations $(\widetilde X_k)_{k=0,\ldots,N}$  is recursively defined by~\eqref{eq:hattilde-k}: 
\[
 \widetilde X_{k}  =  {\cal E}_k(\widehat X_{k-1},Z_{k}),\;     \widehat X_{k-1}  = {\rm Proj}_{\Gamma_{k-1}}(\widetilde X_{k-1}), \; k=1,\ldots,n,   \;  \widetilde X_0  = \bar X_0 ,
\]
where $(Z_k)_{k=1,\ldots,n}$ is an i.i.d. sequence  of   ${\cal N} (0;I_q)$-distributed random vectors,  independent of $ \bar X_0$.

By observing the above recursion and the optimization problem~\eqref{EqOptQuant}, we see that the optimal grids $\Gamma_k$ are defined {\em recursively} (which is not the case for regular marginal quantization developed in former works). At this point the interesting fact which motivates the whole approach is that the distortion function at time $k+1$, $\widetilde D_{k+1}$, $k=0,\ldots,n-1$, can in turn be  written recursively from the grid $\Gamma_k$ already optimized at time $k$ and the distortion function of Normal distribution. 

\smallskip Let $D^{m,\Sigma}_N$ be the distortion function of the ${\cal N}\big(m;\Sigma\big)$-distribution defined on $(\mathbb R^d)^{N}$ by
\[
D^{m,\Sigma}_N(x)  = \mathbb E \min_{1\le i\le N}|\Sigma Z+m-x_i|^2.
\]
As $\widetilde X_k$ has already been (optimally) quantized by a grid $\Gamma_k=\{x^k_1,\ldots, x^k_{N_k}\}$ to which are attached the weights 
\[
p^k_i =\mathbb P(\widehat X_k =x_i^{k}),\; i=1,\ldots, N_k,
\]
 we derive from~\eqref{eq:hattilde-k} that, for every $x^{k+1}\!\in (\mathbb R^d)^{N_{k+1}}$,  
     \begin{eqnarray*}
  \widetilde D_{k+1} (x^{k+1}) & = &  \mathbb E \big[   {\rm dist}({\cal E}_k(\widehat X_k,Z_{k+1}), x^{k+1})^2 \big] \\
  &  = &  \sum_{i=1}^{N_{k}}  \mathbb E \big[   {\rm dist}({\cal E}_k(x_i^{k},Z_{k+1}),x^{k+1})^2 \big] \mathbb P(\widehat X_k =x_i^{k})\\
  &=&  \sum_{i=1}^{N_{k}} p^k_i D_{N_{k+1}}^{m_i^k,\Sigma^k_i}(x^{k+1})
 \end{eqnarray*}
since      ${\cal E}_k(x_i^{k},Z_{k+1}) \stackrel{d}{=} {\cal N}\big(m^k_i;\Sigma^k_i\big)$ with
\[
m^k_i = x_i^k +\Delta b(t_k,x_i^k) \quad \mbox{ and }\quad \Sigma^k_i = \sqrt{\Delta} \sigma(t_k,x^k_i) ,\; k=0,\ldots,n.
\]
Then, owing to Proposition~\ref{PropDifferentiability},   the distortion  function  $\widetilde D_{k+1} $   is continuously differentiable as a function of the $N_{k+1}$-tuple $x^{k+1}$ with pairwise distinct components. This follows from the fact that  the  distortion functions $D^{m, \Sigma}_{N_{k+1}}$ are differentiable  at such a   $N_{k+1}$-tuple  as soon as $\Sigma\Sigma^*$ is positive and definite.

Its gradient is given by 
\begin{eqnarray}
 \nonumber  \nabla  \widetilde D_{k+1}(x^{k+1}) 
   & = & 2 \Big(  \sum_{i=1}^{N_k} p^k_i   \frac{\partial D_{N_{k+1}}^{m^k_i,\Sigma^k_i} (x^{k+1})}{\partial x^{k+1}_j}\Big)_{j=1,\ldots,N_{k+1}}\\
  \label{EqStation}  &=&    2 \Bigg(   \mathbb E \left[  \sum_{i=1}^{N_{k}} p^k_i \Big(   \mbox{\bf{1}}_{\{ {\cal E}_k(x^{k}_i,Z_{k+1})  \in C_j(\Gamma_{k+1})  \}}  \big( x^{k+1}_j   - {\cal E}_k(x^{k}_i,Z_{k+1})  \big) \right] 
     \Bigg)_{j=1,\ldots,N_{k+1}} .
\end{eqnarray}

\begin{rem}  If  $\Gamma_{k+1}$ is a quadratic optimal $N_{k+1}$-quantizer for $\widetilde X_{k+1}$ and  if $\widehat X^{\Gamma_{k+1}}_{k+1}$ denotes   the  quantization of $\widetilde X_{k+1}$ by  the grid $\Gamma_{k+1}$, then  $\nabla \widetilde D_{k+1}(\Gamma_{k+1}) = 0$ and  $\Gamma_{k+1}$ is a stationary quantizer, for $\widetilde X_{k+1}$ $i.e.$  $ \mathbb E\Big(\widetilde X_{k+1} \big \vert  \widehat X_{k+1} \Big)  = \widehat X_{k+1}  $ or equivalently $\Gamma_{k+1}= \{x^{k+1}_i, \, i=1,\ldots,N_{k+1}\}$ with 
\begin{eqnarray}  
 x^{k+1}_{j}  
 & = & \frac{ \sum_{i=1}^{N_k}  p^k_i \,\mathbb E \Big( {\cal E}_k(x_i^{k},Z_{k+1})  \mbox{\bf 1}_{\{ {\cal E}_k(x_i^{k},Z_{k+1}) \in C_j(\Gamma_{k+1}) \} }\Big)      }{p^{k+1}_j}  \label{EqLloyd1}
\\
 \mbox{and } \qquad p^{k+1}_j&=&   \sum_{i=1}^{k} p^k_i \, \mathbb P \big(   {\cal E}_k(x_i^{k},Z_{k+1})  \in C_j(\Gamma_{k+1})   \big)  ,\; j=1,\ldots, N_{k+1}. \label{EqLloyd2} 
   \label{EqLloyd2}
 \end{eqnarray}
\end{rem}

\noindent {\sc Application to the computation of the grids}.   It follows from the stationarity Equations~\eqref{EqStation} on the one hand and \eqref{EqLloyd1}  and~(\ref{EqLloyd2}) on the other hand that one can compute by a {\em zero search} stochastic gradient descent  (known as CLVQ) or a randomized {\em fixed point}  Lloyd algorithm time in a  step by time step forward induction. In both cases, we are not only interested in the grids but by the whole distribution of $\widetilde X_k$ so we need to implement a companion procedure to compute the weights $p^k_i$. Put in a different way, this simply means that the distribution of the random vectors $\widetilde X_{k}$ can be simulated recursively. 

However, in view of our applications, we need much more than these distributions: though the sequence $(\widehat X_k)_{k=0,\ldots,n}$ is not a Markov chain we  need all  its transitions ${\cal L}\big(\widehat X_{k+1}\,|\, \widehat X_k\big) $, $k=0,\ldots,n-1$, namely  
\begin{eqnarray}
\nonumber p^k_{ij}& = & \mathbb P\big(\widehat X_{k+1}\!\in C_j(\Gamma_{k+1})\vert \widehat X_{k} \!\in C_i(\Gamma_{k})\big) = \mathbb P\big(\widetilde X_{k+1}\!\in C_j(\Gamma_{k+1})\vert \widetilde X_{k} \!\in C_i(\Gamma_{k})\big)\\
\label{transitions} &=&   \mathbb P \big(   {\cal E}_k(x_i^{k},Z_{k+1})  \in C_j(\Gamma_{k+1})   \big), \;\quad i\!\in \{1,\ldots,N_k\},\; j\!\in \{1,\ldots,N_{k+1}\}.
 \end{eqnarray}

 Looking back at~\eqref{EqLloyd2}, it is clear that, whatever the adopted method is,  these quantities are  to be computed in the above procedure to have access to $p^{k+1}_j$ (this is but Bayes' formula).
 
 \medskip This optimization phase which is crucial for applications can become time consuming as the dimension $d$ grows since it relies {\em in fine} on Monte Carlo simulations based on nearest neighbor searches. In fact when $d$ is larger than $3$ or $4$ a variant should be devised to make the procedure reasonably fast (see~\cite{PagSagMQ2}). In fact all the above formulas can be expressed as $d$-dimensional integrals over (convex) Voronoi cells.  In dimension $1$, $2$ or possibly $3$ it can be still possible to compute some of such  integrals (see Qhull website: {\tt www.qhull.org}).  
%
%
%

\smallskip In the section below, we detail a $1$-dimensional version which turns out to be extremely fast in practice, since it take advantage of the  Newton-Raphson algorithm to perform a  zero search procedure for $\nabla \widetilde D_{k}$.


\subsection{The one dimensional  setting}
In a dimensional setting, we can canonically represent a grid $\Gamma=\{x_1,\ldots,x_{_N}\}$ of full size $N$ by a unique $N$-tuple with increasing components $(x_1,\ldots,x_{_N})$. So from now on in this section, we will always make the abuse of notation that the symbol $\Gamma$ will denote the $N$-tuple of its values in increasing order. 

\subsubsection{Computing marginal quantizers   with Newton-Raphson algorithm}

It is a well-known fact that if a distribution $\mu$ has a continuous density the resulting distortion functions $D^{\mu}_N$ are in fact twice differentiable at $N$-tuples with pairwise distinct components and negligible Voronoi diagram boundary, with an explicit, though rather involved,  expression (see $e.g.$~\cite{PagYu}). So, as the distortion functions $D_{N}^{m,\Sigma}$ associated to the ${\cal N}(m;\Sigma)$ are twice differentiable, in turn the distortion function of $\widetilde D_{k+1}$ is twice differentiable. Hence  it is possible, at least formally, to write down  a Newton-Raphson zero search procedure, indexed by $\ell\!\in \mathbb N$, where  a current grid $\Gamma_{k+1}^{(\ell)}$ is updated as follows:
\begin{equation}
\Gamma_{k+1}^{(\ell+1)}   =    \Gamma_{k+1}^{(\ell)} - \big(\nabla^2  \widetilde D_{k+1} (\Gamma_{k+1}^{(\ell)}) \big)^{-1} \nabla \widetilde D_{k+1}( \Gamma_{k+1}^{(\ell)}), \quad \ell\ge 0,
\end{equation}
starting  from a $\Gamma^{(0)}\!\in \mathbb R^{N_{k+1}}$ (with increasing components).  Of course,  $ \nabla \widetilde D_{k}(\Gamma_k) $ and $ \nabla^2 \widetilde D_{k}(\Gamma_k) $ denote  respectively the gradient vector and the Hessian matrix of the distortion function   $\widetilde D_{k}$. 

To this end we  need  a closed form for the Hessian $\nabla^2 \widetilde D_{k+1}$ of $\widetilde D_{k+1}$. We will rely on the expression~\eqref{EqStation} for the gradient $\nabla\widetilde D_{k+1}$ and take advantage of the fact that it can  be reduced to the gradient  of the distortion function of a Gaussian   $m+\sqrt{v}Z$, $Z\stackrel{d}{=}{\cal N}(0;1)$. We will  denote by $\Phi_0$ and $\Phi_0'$ the  cumulative distribution function and the  probability density   function of the ${\cal N}(0;1)$ distribution  respectively and we will extensively take advantage of the obvious fact 
\[
\int_{-\infty}^{x} \xi \Phi_0'(\xi)d\xi = -\Phi'_0(x),\; x\!\in \mathbb R.
\]

%


To simplify notation,  set  for every $k=0,\ldots,n-1$ and every    $j=1,\ldots,N_{k+1}$,
$$ 
x^{k+1}_{j-1/2} = \frac{ x^{k+1}_{j} + x^{k+1}_{j-1}  }{2}, \  x^{k+1}_{j+1/2} = \frac{ x^{k+1}_{j} + x^{k+1}_{j+1}  }{2}, \ \textrm{ with }  x^{k+1}_{1/2} = -\infty,  x^{k+1}_{N_{k+1}+1/2} =+\infty, 
$$ 
and let us define for every $\xi\!\in \mathbb R$, $v_k(\xi) = \sqrt{\Delta} \sigma_k(\xi)$, $m_k(\xi) = x +\Delta b_k(\xi)$, 
$$ 
x^{k+1}_{j-}(\xi): = \frac{x^{k+1}_{j-1/2} - m_{k}(\xi) }{v_k(\xi)} \quad \textrm{ and } \quad x^{k+1}_{j+}(\xi): = \frac{x^{k+1}_{j+1/2} - m_{k}(\xi) }{v_k(\xi)},\; k=0,\ldots,n-1.
$$
Let    $\Gamma_{k+1}=(x^{k+1}, \ldots, x^{k+1}_{N_{k+1}}) \!\in \mathbb R^{N_{k+1}}$ (we temporarily drop the dependence of all the grid points 
in $\ell $ to alleviate notations).  Also have in mind that $p^k_i = \mathbb P(\widehat X_k = x^k_i)$. 
%
Then, for every  $j=1,\ldots,N_{k+1}$
\setlength\arraycolsep{0.1pt}
\begin{eqnarray*}
  \frac{\partial \widetilde D_{k+1}(\Gamma_{k+1}) }{\partial  x^{k+1}_j} & = &  \sum_{i=1}^{N_k} p^k_i\Big[ \big( x^{k+1}_j - m_k(x^{k}_i)\big)  \Big(  \Phi_0( x^{k+1}_{j+}(x^{k}_i))  -  \Phi_0(x^{k+1}_{j-}( x^{k}_i)) \Big) \\
  &  & \qquad  + \,  v_k( x^{k}_i) \big(\Phi_0'(  x^{k+1}_{j+}(  x^{k}_i)) - \Phi_0'( x^{k+1}_{j-}(x^{k}_i)) \big) \Big].
   \end{eqnarray*}
 The  diagonal terms of the Hessian matrix $\nabla^2  \widetilde D_{k+1} (\Gamma_{k+1}) $ are given by:
\begin{eqnarray*}
  \frac{\partial^2 \widetilde D_{k+1}(\Gamma_{k+1}) }{\partial^2  x^{k+1}_j} = \sum_{i=1}^{N_k} &  p^k_i \Big[ & \Phi_0( x^{k+1}_{j+}( x^{k}_i))  -  \Phi_0( x^{k+1}_{j-}( x^{k}_i)) 
  \\
&  & - \frac{1}{4 v_k( x^{k}_i)} \Phi_0'( x^{k+1}_{j+}( x^{k}_i))( x^{k+1}_{j+1} - x^{k+1}_{j}) \\
&  &-  \frac{1}{4 v_k( x^{k}_i)} \Phi_0'(x^{k+1}_{j-}( x^{k}_i))( x^{k+1}_{j} - x^{k+1}_{j-1}) \Big] 
\end{eqnarray*}
and its  sub-diagonal terms are  
 \begin{equation*}
 \frac{\partial^2 \widetilde D_{k+1}(\Gamma_{k+1}) }{\partial  x^{k+1}_j \partial  x^{k+1}_{j-1} } = -\frac{1}{4} \sum_{i=1}^{N_k}  p^k_i \frac{1}{v_k( x^{k}_i)} ( x^{k+1}_j - x^{k+1}_{j-1}) \Phi_0'( x^{k+1}_{j-}( x^{k}_i)).
 \end{equation*}
 The upper-diagonals terms are
  \begin{equation*}
\frac{\partial^2 \widetilde D_{k+1}(\Gamma_{k+1}) }{\partial  x^{k+1}_j \partial x^{k+1}_{j+1} } = -\frac{1}{4} \sum_{i=1}^{N_k} p^k_i \frac{1}{v_k( x^{k}_i)} ( x^{k+1}_{j+1} -  x^{k+1}_{j}) \Phi_0'( x^{k+1}_{j+}( x^{k}_i)).
\end{equation*}

A similar  idea  combining  (vector or functional) optimal quantization with Newton-Raphson zero search procedure is used in~\cite{FriSag} in  a variance reduction context  as an alternative and robust method to simulation based recursive importance sampling procedure to estimate the optimal change of measure.  Furthermore,   the convergence of the modified  Newton-Raphson algorithm to the optimal quantizer is shown in the framework of~\cite{FriSag} to be bounded by the quantization error. However, the  tools used to show it   do  not apply directly  in our context and the proof  of the convergence of our  modified Newton algorithm to an optimal quantizer remains  an open question.

\subsubsection{Computing the weights and the transition probabilities}

Once we have access  to   the quadratic optimal quantizers $\Gamma_{k}$ of the marginals $\widetilde X_{k}$, for  $k=0,\ldots, n$   (which are estimated   using the Newton-Raphson  algorithm described previously)  we  need to  compute the associated weights $\mathbb P(\widetilde X_{k}\!\in C_j(\Gamma_k))$, $j =1,\ldots,N_{k}$, for $k=0,\ldots,n$ or the transition probabilities $\mathbb P(\widetilde X_{k}\!\in C_j(\Gamma_k)\vert \widetilde X_{k-1} \!\in C_i(\Gamma_{k-1}))$, $i=1, \ldots,N_{k}$, $j=1,\ldots,N_{k+1}$.  We show in the next result how to compute  them.

\begin{prop} 
Let $\Gamma_{k+1}$ be a quadratic optimal quantizer for the marginal random variable $\widetilde X_{k+1}$. Suppose that  the quadratic optimal quantizer $\Gamma_k$ for $\widetilde X_{k}$ and its companion  weights $\mathbb P(\widetilde X_{k}\!\in C_i(\Gamma_k))$, $i=1,\cdots,N_k$, are computed. 

\begin{enumerate}
  \item The transition probability $p^k_{ij}= \mathbb P\big(\widetilde X_{k+1}\!\in C_j(\Gamma_{k+1})\vert \widetilde X_{k} \!\in C_i(\Gamma_{k})\big)$ is given by
  \setlength\arraycolsep{3pt}
\begin{eqnarray}
 p^k_{ij}& =&p^k_i  \Big(\Phi_0(x^{k+1}_{j+}(x^{k}_i))  -   \Phi_0(x^{k+1}_{j-}(x^{k}_i))  \Big) .
  \label{EqEstProba}
\end{eqnarray}
\item  The probability $p^{k+1}_j= \mathbb P\big(\widetilde X_{k+1}\!\in C_j(\Gamma_{k+1})\big)$  is given for every $j=1,\cdots,N_{k+1}$ by 
\setlength\arraycolsep{3pt}
\begin{eqnarray}
p^{k+1}_j & =&  \sum_{i=1}^{N_k} p^k_i \Big(   \Phi_0(x^{k+1}_{j+}(x^{k}_i))  - \Phi_0(x^{k+1}_{j-}(x^{k}_i))  \Big) .
  \label{EqEstProba}
\end{eqnarray}

  \end{enumerate}

\end{prop}

\begin{proof} [{\bf Proof.}]
1. For every $k\!\in \{ 1,\ldots,n-1 \}$, for every $i=1, \ldots, N_k$ and  for every  $j=1,\ldots,N_{k+1}$, we have
\setlength\arraycolsep{3pt}
 \begin{eqnarray*} \label{EqSumNk-1Dist}
 \mathbb P(\widetilde X_{k+1}\!\in C_j(\Gamma_{k+1})\vert \widetilde X_{k} \!\in C_i(\Gamma_{k})) &=&   \mathbb P \big(\widetilde X_{k+1}\!\in C_{j}(\Gamma_{k+1}) \vert \widehat  X_{k} =  x_i^{N_{k}}  \big) \\
  & =  & \mathbb P \big(\widetilde X_{k+1} \leq  x^{N_{k+1}}_{j+1/2}\vert  \widehat  X_{k} =  x_i^{N_{k}} \big)   - \  \mathbb P \big(\widetilde X_{k+1} \leq  x^{N_{k+1}}_{j-1/2} \vert   \widehat  X_{k} =  x_i^{N_{k}} \big)\\
  & = &  \Phi_0(x^{k+1}_{j+}(x^{k}_i))  -   \Phi_0(x^{k+1}_{j-}(x^{k}_i)). 
 \end{eqnarray*}

\noindent 2. We have for every $k\!\in \{ 1,\ldots,n-1 \}$ and for every  $j=1,\ldots,N_{k+1}$, 
 \setlength\arraycolsep{3pt}
\begin{eqnarray*} \label{EqSumNk-1} 
 \mathbb P \big(\widetilde X_{k+1}\!\in C_{j}(\Gamma_{k+1}) \big) & = &  \mathbb E\big[ \mathbb P\big(\widetilde X_{k+1}\!\in C_{j}(\Gamma_{k+1})   \vert \widehat X_k \big) \big]  \\
 & = &    \sum_{i=1}^{N_{k}} \mathbb P \big(\widetilde X_{k+1}\!\in C_{j}(\Gamma_{k+1})   \vert \widehat  X_{k} =  x_i^{N_{k}} \big)  \mathbb P(\widetilde X_k\!\in C_i(\Gamma_k)).
 \end{eqnarray*}
 Now, il follows from the first assertion  that   
 \begin{equation*} 
  \mathbb P \big(\widetilde X_{k+1}\!\in C_{j}(\Gamma_{k+1}) \vert \widehat  X_{k} =  x_i^{N_{k}}  \big)  =   \Phi_0\big(x^{k+1}_{j+}(x^{k}_i)\big)  -   \Phi_0\big(x^{k+1}_{j-}(x^{k}_i)\big). 
 \end{equation*}
This completes the proof. 
\end{proof}

%
%
\section{Numerical examples} \label{SectNum}
  \subsection{Numerical example for Brownian motion} 
  
%

  We consider a real valued Brownian motion $(W_t)_{t \in [0,1]}$ and quantize the random variable $W_1$ by both regular marginal quantization and recursive marginal quantization methods.   Denote by $D_{M}^{\rm reg}(\Gamma)$   the regular quantization distortion associated to $W_1$ and, for a given discretization   mesh: $t_0=0< \cdots < t_n=1$ of size $n$,  denote  by $D^{\rm rec}_n(\Gamma_n)$ the  distortion associated to the recursive quantization of $W_1$ at the end point $t_n=1$.

We recall that for a given grid size $M$,
\[
D_{M}^{\rm reg} (\Gamma)  =   \mathbb E\vert W_1 - \widehat W_1^{\Gamma} \vert^2,
\] 
 where  the  optimal grid $\Gamma$ for the regular  marginal quantization  is obtained  by solving (using Newton-Raphson algorithm) the following minimization problem: $\!\inf_{\Gamma\!\in \mathds R^M}  D_{M}^{\rm reg} (\Gamma)$, which corresponds to the optimal grid of the standard Gaussian distribution. On the other hand, 
\[
D^{\rm rec}_n(\Gamma_n) =  \mathbb E\vert W_1 - \widehat W_1^{\Gamma_n} \vert^2,
\]
where $\Gamma_n$ is the $n$-th component of the sequence of recursive marginal  quantization grids $(\Gamma_k)_{k=0, \ldots, n}$   defined for every $k=0,\ldots,n$ by 
\[
\Gamma_k \!\in  \argmin \{ \widetilde D_{k}(\Gamma), \Gamma \subset \mathbb R, {\rm card}(\Gamma) \leq N_{k} \}, 
\]
where 
     \begin{eqnarray*}
  \widetilde D_{k} (\Gamma)  :=    \mathbb E \big[   {\rm dist}(\widetilde W_{t_{k}}, \Gamma)^2 \big] & = & \mathbb E \big[   {\rm dist}\big(\widehat W_{t_{k-1}}^{\Gamma_{k-1}} + \sqrt{\Delta} Z_{k} , \Gamma \big)^2 \big] \\
  &  = &  \sum_{i=1}^{N_{k-1}}  \mathbb E \big[   {\rm dist}\big(w_i^{k-1} + \sqrt{\Delta} Z_{k}, \Gamma \big)^2 \big] \mathbb P\big(\widehat W_{t_{k-1}}^{\Gamma_{k-1}} =w_i^{k-1} \big)
 \end{eqnarray*}
and  $\widehat W_{t_{k-1}}^{\Gamma_{k-1}}$ is defined from the following recursion: $\widehat W_{0}^{\Gamma_{0}} = 0$ and for $\ell =1,\ldots,k-1$,
 \[
    \widetilde W_{t_{\ell}}  =  \widehat W_{t_{\ell -1}}^{\Gamma_{\ell-1}} + \sqrt{\Delta} Z_{\ell}    \  \textrm{ and } \  \widehat W_{t_{\ell}} ^{\Gamma_{\ell}} = {\rm Proj}_{\Gamma_{\ell}}(\widetilde W_{t_{\ell}}),  \  (Z_{\ell})_{\ell=1,\ldots, k} \textrm{ is  }  i.i.d.,  \textrm{ and }{\cal N}(0;1) \textrm{-distributed}.  
\] 

 Our aim is precisely to  compare   the regular  quantizations  error $(D_{M}^{\rm reg} (\Gamma))^{1/2}$  and  the  recursive quantizations error  $(\widetilde D_{n}^{\rm rec} (\Gamma_n))^{1/2}$  at time $t_n$ of size $N_n$, for the same  grid size  $M=N_n$ (the choice  of the grid size  $N_n$ depends on the used dispatching procedure and is specified  further).  To compute the quantity $D^{\rm rec}_n(\Gamma_n)$, we need to know the distribution of the couple of random variables  $(W_1, \widehat W_1^{\Gamma_n})$ or $(W_1, \widetilde W_1)$, which distributions are unfortunately  unknown. 
 
  This  leads us to consider the following approximation for  $D^{\rm rec}_n(\Gamma_n)$:
  \begin{equation*} \label{EqApproxDistrec}
  D^{\rm rec}_n(\Gamma_n) =\sum_{i=1}^{N_n} \mathds E \Big(\vert  W_1  - w_i^n \vert^2 \mbox{\bf{1}}_{ \{\widetilde W_1  \in C_i(\Gamma_n) \}} \Big)\approx \widetilde D^{\rm rec}_n(\Gamma_n):= \sum_{i=1}^{N_n} \mathds E \Big(\vert  W_1  - w_i^n \vert^2 \mbox{\bf{1}}_{\{W_1  \in C_i(\Gamma_n)  \}} \Big).
  \end{equation*}
Therefore, we compare in practice the two quantities $(\widetilde D^{\rm rec}_n(\Gamma_n))^{1/2}$ and  $(D_{M}^{\rm reg} (\Gamma))^{1/2}$. To this end, we fix first  the  mesh size $n=50$ and   we make the total budget   $N=N_1+ \cdots +N_n$ (given that $N_0=1$) of  the recursive  grid sizes varying  $50$ by $50$, from $250$ up to $5000$.  We choose the sizes $N_k$  following two procedures: the  optimal and  the uniform  dispatching. 

%
 
  \smallskip
  $\rhd$ {\em Uniform dispatching}.  For  a given global budget $N$,  we make an ``equal grid size dispatching'' by choosing $N_k = N/n$, for $k=1,\ldots,n$. If for example  $N=N_1+ \cdots+N_{n} = 250$, we will have $N_k  = 5$, for every $k\!\in \{1, \ldots, n\}$. For comparison tool with  the regular quantization method we choose  the size $M$ of the regular quantization equal to $N_n$, for every fixed global budget  $N$. The comparaison result  is depicted on the right hand side graphic of Figure~\ref{figCompRegMar}. In this figure,  we plot the (approximated) recursive marginal quantization errors $(\widetilde{D}_{n}^{\rm rec} (\Gamma_{n}))^{1/2}$, for $\vert \Gamma_{n} \vert=5, \ldots,100$, and the regular quantization errors $(D_M^{\rm reg}(\Gamma))^{1/2}$, for $M=\vert \Gamma \vert =5, \ldots,100$.

 \smallskip
  $\rhd$ {\em  Optimal dispatching}.  In this case, the  sizes $N_k$ are obtained from the optimal dispatching procedure described  in the {\sc Practionner corners}, (a). First of all, we have to choose  the coefficients $a_{\ell}$ (appearing in Theorem~\ref{TheoConvergenceRate})  corresponding to  the Brownian motion. Following, step by step, the proof of  Theorem~\ref{TheoConvergenceRate} and setting $\eta =1$ (keep in mind that in the Brownian case  $x_0 = 0$), we may improve our estimate for the coefficient $a_{\ell}$. Namely, we have (at time $t_n=T$)
\[  
a_{\ell} = \left[  \sqrt{\frac{2}{\pi}}(4 + \sqrt{\Delta}) (e^{2 t_{\ell}} -1)   \right]^{1/3}, \quad \ell =1,\ldots,n.
\]
Making  $N$ vary  from  $250$ to $5000$, the optimal dispatching leads to the following  sizes for the grid $\Gamma_n$:   
\[
\vert \Gamma_n \vert\!\in {\cal G} =  \{ 6,    8,    9,    10,    11,    13,   \cdots, 123,   124,   126,  127 \}.
\]
Notice that we just display  the first and the last ones   values  of   ${\cal G}$ because it is of size  $96$.  Here is how to read these values.  If we fix the size of the global budget to  $N=250$,  then the optimization procedure described in  the {\sc Practionner corners}, (a), will generate the optimal grid sizes $N_1, N_2, \ldots, N_n$ inducing the minimal error bound  in Theorem~\ref{TheoConvergenceRate}. In the set ${\cal G}$, we only display the size $N_n=N_{50}$ of the terminal grid $\Gamma_{50}$ which then is equal to $6$ for $N=250$. So, to compare the regular quantization method and  the recursive procedure with optimal dispatching  and global budget $N=250$, we have to put the grid size $M$ of the regular quantization method to $M=6$.  If the global budget  $N =300$ then we  choose  $M=N_n=8$, etc, and  if $N=5000$ then we choose $M=N_n =127$.

The comparaison result is depicted  on the left hand side of Figure~\ref{figCompRegMar}. So, this figure plots  the (approximated) recursive quantization errors $(\widetilde{D}_{n}^{\rm rec} (\Gamma_{n}))^{1/2}$, where  $\vert \Gamma_{n} \vert$ is given by the optimal dispatching procedure, and the  regular quantization errors $(D_M^{\rm reg}(\Gamma))^{1/2}$   for $M=\vert \Gamma_n \vert\!\in {\cal G}$.

 \begin{figure}[htpb]
  \!\includegraphics[width=8.5cm,height=7.0cm]{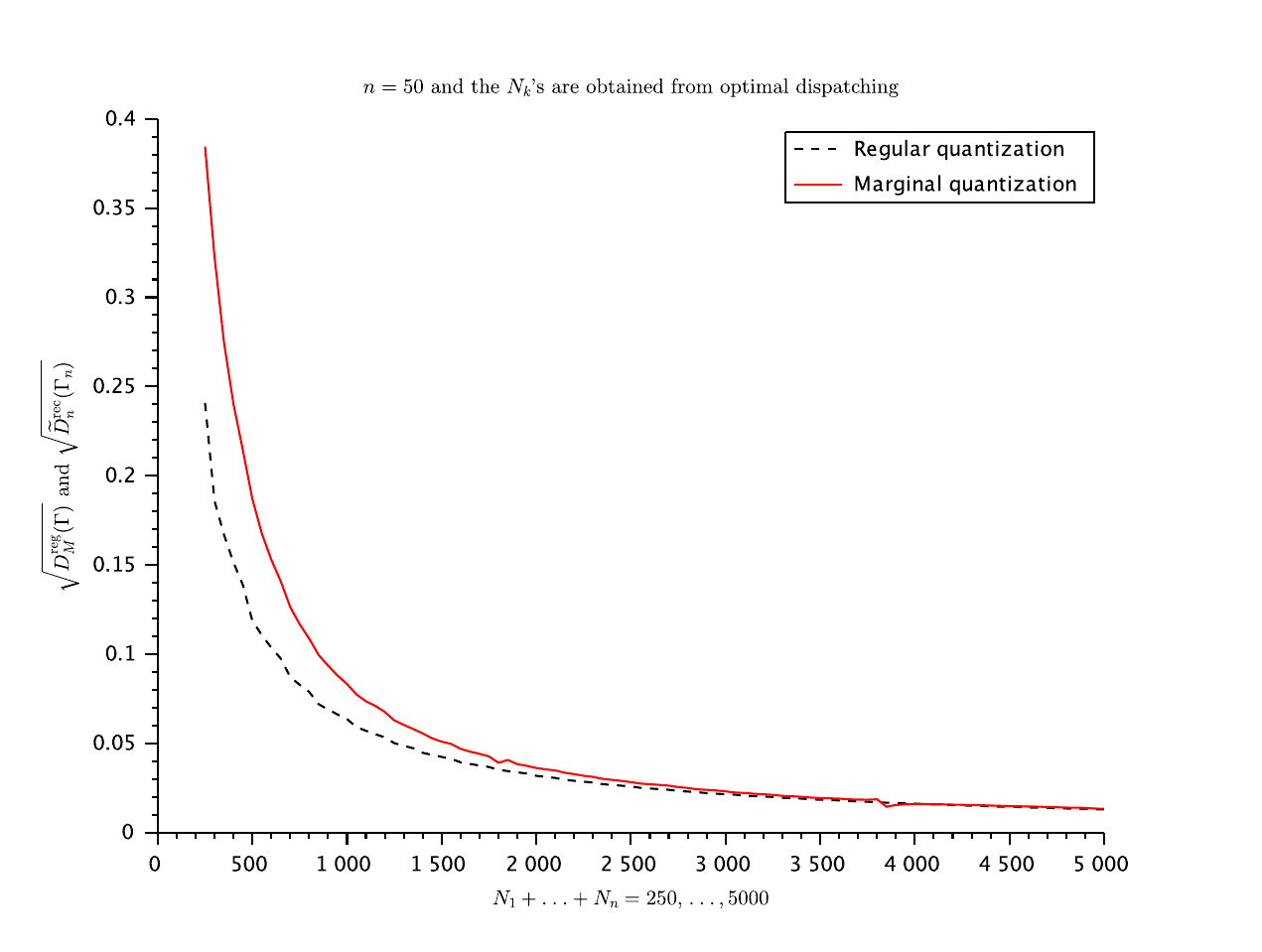}	
\hfill   \!\includegraphics[width=8.5cm,height=7.0cm]{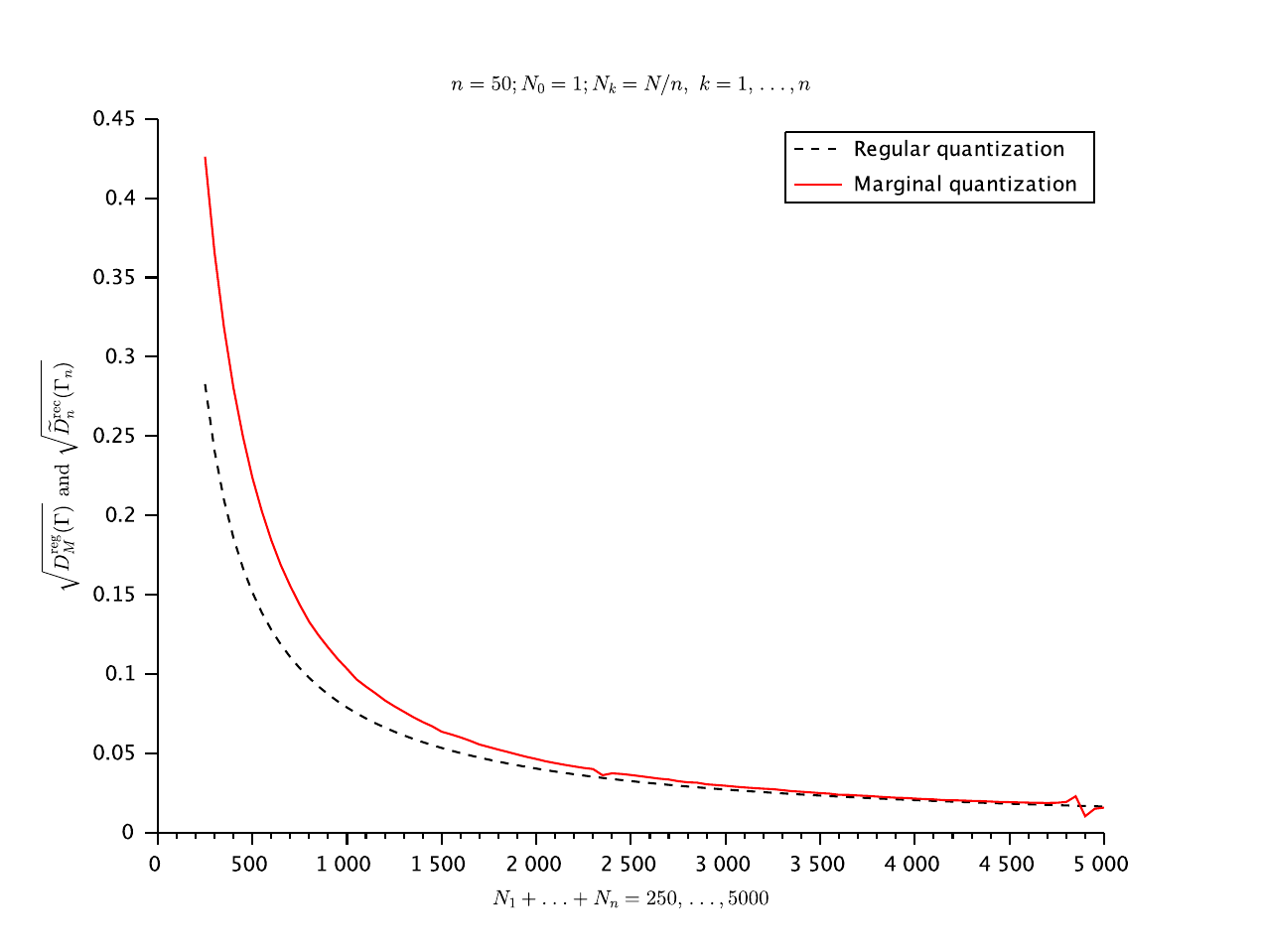}
  \caption{\footnotesize  Comparison of the regular quantization (Regular quantization)  of $W_1$  with its recursive  marginal quantization (Marginal quantization) (where $W$ is a Brownian motion).  Abscissa axis: $n=50$ and  the total budget  $N=N_1+\cdots+N_n$  varies from $250$ up to $5000$.  Ordinate axis:  For  a given $N$,  $(a)$ (left hand side graphic)  we depict  the recursive  quantization  errors $(\widetilde{D}_{n}^{\rm rec} (\Gamma_{n}))^{1/2}$ where  $\vert \Gamma_{n} \vert$ is given by the optimal dispatching procedure, and the regular quantization errors $(D_M^{\rm reg}(\Gamma))^{1/2}$   for $M=\vert \Gamma_n \vert$; $(b)$ (right hand side graphics) we set $N_k = N/n$, for $k=1,\ldots,n$ and depict the recursive quantization  errors $(\widetilde{D}_{n}^{\rm rec} (\Gamma_{n}))^{1/2}$, for $\vert \Gamma_{n} \vert=5, \ldots,100$, and the regular quantization errors $(D_M^{\rm reg}(\Gamma))^{1/2}$, for $M=\vert \Gamma \vert =5, \ldots,100$. } 
  \label{figCompRegMar}
\end{figure}

\begin{figure}[htpb]
 \begin{center}
  \!\includegraphics[width=11.5cm,height=8.0cm]{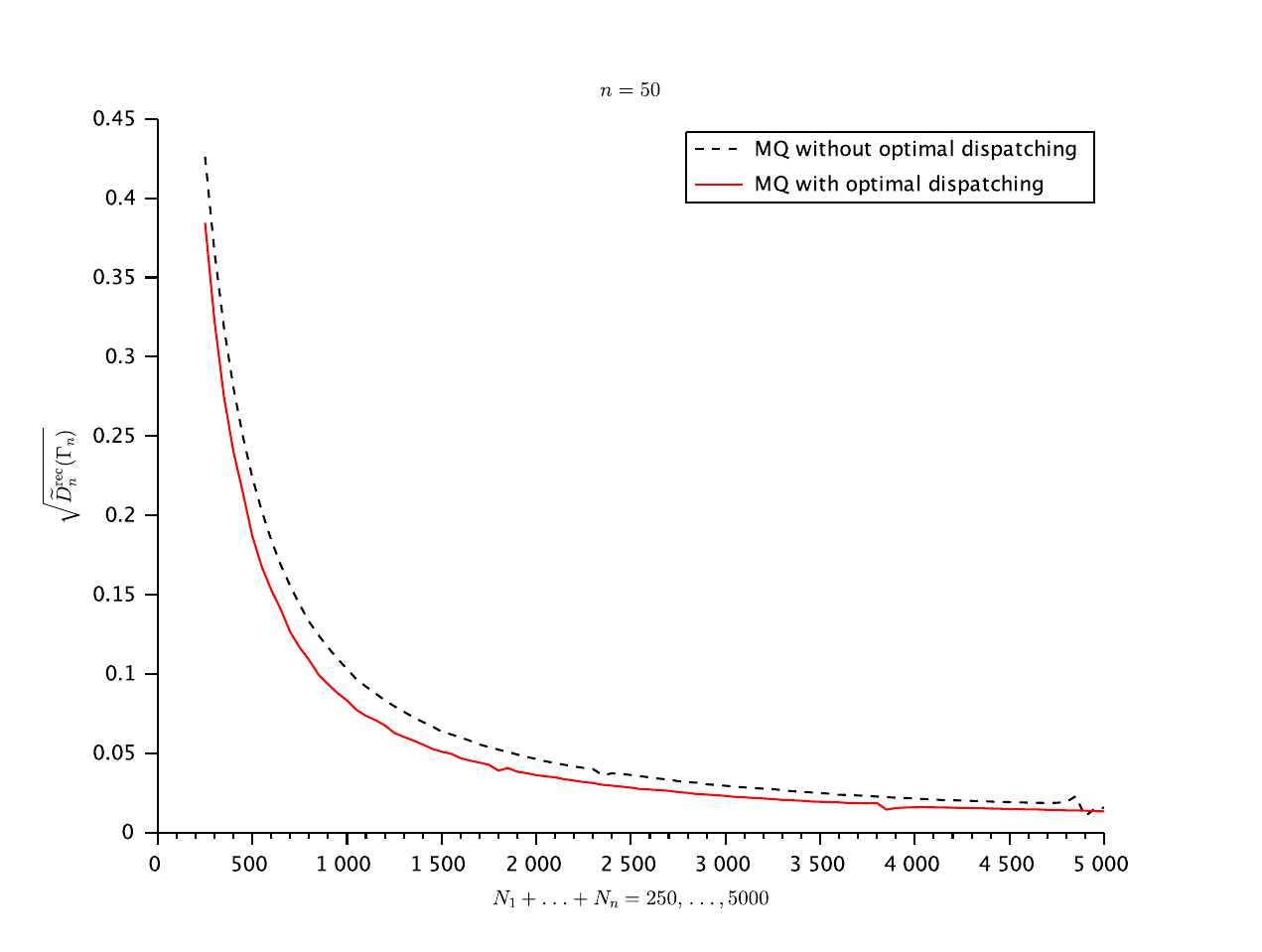}	
  \caption{\footnotesize  Comparison of the recursive marginal quantization errors. Abscissa axis:  $n=50$ and  the total grid sizes  $N=N_1+\cdots+N_n$  varies from $250$ up to $5000$. Ordinate axis: on one hand, we set $N_k = N/n$, for $k=1,\ldots,n$ and depict   the  (approximated) recursive  quantization errors $(\widetilde{D}_{n}^{\rm rec} (\Gamma_{n}))^{1/2}$, for $\vert \Gamma_{n} \vert=5, \ldots,100$ (MQ without optimal dispatching), and, on the other hand,  we depict  the (approximated) recursive  quantization errors $(\widetilde{D}_{n}^{\rm rec} (\Gamma_{n}))^{1/2}$  where  $\vert \Gamma_{n} \vert$ is given by the optimal dispatching procedure (MQ with optimal dispatching).} 
  \label{figCompMQ}
 \end{center}
\end{figure}

\smallskip
 {\em Conclusion}.  The graphics of figure~\ref{figCompRegMar} and~\ref{figCompMQ} suggest  two observations. The first one is that, when the global budget $N$ is too small,  the regular quantization of $W_1$ is, as expected,  more efficient than  both  recursive marginal  quantizations (without and  with optimal dispatching). But, when this  global budget $N$ increases,   the absolute error between  regular and  recursive quantization methods  fades  and becomes close to zero up to $10^{-2}$. 
 

The second conclusion is that the  recursive marginal quantization method with the optimal dispatching of the grid size over discretization time steps  outperforms a setting where  the  grids are  of equal sizes.  However, when $N$  increases, the  recursive marginal quantization with optimal dispatching becomes, as expected, more time consuming while  both methods  yield  almost the same results (the errors are equal up to $10^{-2}$).  

\smallskip

In the next section we propose  an application of our  method to  the pricing of European options in a  local volatility models. We remark that when using the marginal quantization methods,  we have to choose a  big global budget $N$ to reach  good price estimates. Like in the Brownian case, numerical results show that both  recursive marginal quantization  methods (with  optimal and uniform dispatching) lead to the same price estimates (up to $10^{-3}$) whereas  the complexity of the optimal dispatching method becomes higher (as pointed out in the practitioner's corner).   This is why we will use in the following section the recursive marginal quantization method with uniform dispatching  procedure.

\subsection{Pricing of European options in a local volatility model}  
\subsubsection {The model}
We consider a pseudo-CEV model (see e.g.~\cite{LemPag}) where the dynamics of  the stock price  process is ruled  by the following SDE (under the risk neutral probability)
\begin{equation}
dX_t = rX_t dt + \vartheta \frac{X_t^{\delta+1}}{\sqrt{1+X_t^2}} dW_t, \quad X_{0} = x_{0},
\end{equation}
for some $\delta\!\in (0,1)$ and $\vartheta\!\in (0,\underline{\vartheta}]$   with $\underline{\vartheta}>0$. The parameter $r$  stands for  the interest rate and $\sigma(x) := \vartheta \frac{x^{\delta}}{\sqrt{1+x^2}}$ corresponds to the local volatility function. This model  becomes   very close to  the CEV model,  specially when the initial value of the stock process  $X_0$ is  large enough.  In this  case the local volatility $\sigma(x) \approx \vartheta x^{\delta-1}$.  

\vskip 0.3cm
 
\begin{table}[htbp]
 \begin{center}
 \footnotesize
  \begin{tabular}{l*{6}{l}}
&   &      &    &   &   \\
   
  \hline
&  $\vartheta$ & \qquad   MC ($10^5$) &    RMQ   ($N_k=300$) &  \qquad MC ($10^6$)    &   RMQ ($N_k=400$)     \\
  \hline  
  &   &      &    &    &  &   \\
 
& $0.5 $ & \qquad  \ \  $0.0022 $ &  \quad $0.0016$    & \qquad \quad  $ 0.0018$   & $\ \ 0.0017$    \\
 CI &  &   ${\bf [0.0017;0.0028] }$ &     &  \quad  ${\bf [0.0017;0.0019] }$  &    \\
&  $ 0.6$ & \qquad \ \   $0.0113$ &\quad $0.0108$   &  \qquad \quad  $0.0111$ &  \ \  $0.0110 $     \\
CI &  &  $ {\bf [0.0101;0.0125] } $&      &  $ \quad \ {\bf [0.0107;0.0115] } $  &    \\
&$ 0.7$ &  \qquad \ \ $0.0377 $ &\quad  $0.0367$  &   \qquad \quad  $0.0373$  &  \ \  $0.0370$   \\
 CI &   & $  {\bf [0.0353;0.0401] }$ &   &   $  \quad  \  {\bf [0.0366;0.0381] }$  &       \\
&  $0.8$    &    \qquad \ \   $ 0.0883$   & \quad  $0.0867$   &   \qquad \quad $0.0876$   & \ \ $ 0.0871$      \\
  CI &   &   ${\bf [0.0843;0.0923] }$ &    &  \quad \  ${\bf [0.0863;0.0886] }$ &   \\ 
 & $ 0.9$      &  \qquad \ \   $0.1696$   & \quad  $0.1644$  &  \qquad \quad   $0.1659$ &   \ \  $0.1649$    \\
  CI &   &   ${\bf [0.1635;0.1756] }$ &  &    \quad \  ${\bf [0.1640;0.1678] }$   &  \\
 & $1.0 $    & \qquad \quad   $ 0.267 $   & \quad  $0.270$  &  \qquad \quad \  \   $0.271$  &   \quad  $ 0.271$  \\
  CI &  & \ \  \    ${\bf [0.259;0.275] }$ &   &  \quad  \ \ \    ${\bf [0.269;0.274] }$ &    \\
   &$ 2.0$     & \qquad \quad  $ 2.423 $   &  \quad  $2.425$ &  \qquad \quad  \  \  $2.433$ & \quad $ 2.426$ \\
  CI &  & \ \ \     ${\bf [2.387;2.459] }$ &  &   \quad \ \ \     ${\bf [2.422;2.445] }$   &   \\
   & $ 3.0$     & \qquad \quad  $5.424$   & \quad  $5.475$ &  \qquad \quad   \ \  $5.492$  &  \quad $ 5.478$   \\
  CI &  & \ \  \  ${\bf [5.424;5.512] }$ &   &  \quad \ \ \   ${\bf [5.471;5.512] }$ &   \\
    & $4.0$     & \qquad \quad   $8.893$   & \quad  $8.804$ &  \qquad \quad \  \    $8.806$  & \quad  $ 8.808$ \\
  CI &  &   \ \ \  $ {\bf [8.801;8.986] }$ &    &  \quad \ \ \   ${\bf [8.777;8.835] }$  &    \\
  &   &      &    &   &    \\
 \hline
  \end{tabular} 
\caption{ \label{tabPrixPutLV} \small{ (Pseudo-CEV model) Comparison of the Put  prices obtained from Monte Carlo (MC) simulations (followed by its size) with associated confidence intervals (CI) and from the  recursive marginal quantization (RMQ) method  with  equal grid size  allocation $N_k=300$ or $N_k=400$, $\forall k=1,\ldots,n$. The parameters are:  $r =0.15$; $\delta = 0.5$; $n=120$; $T=1$; $K=100$; $X_{0}=100$;  and for varying values of $\vartheta$.}}
\end{center}
\end{table}

\vskip 0.3cm

\begin{table}[htbp]
 \begin{center}
 \footnotesize
 \begin{tabular}{l*{7}{l}}
    
  \hline
&  $K$ & \qquad  \   MC ($10^5$) & \qquad \ \  MC ($10^6$)   &    RMQ $(N_k=300)$   & \qquad \ \  MC ($10^7$)    &      RMQ  $(N_k=400)$    \\
  \hline  
  &   &      &    &    &  &    \\
& $100$  & \qquad \quad    $08.89$   & \qquad \quad \ \ \  $08.81$ & \quad  $08.80$    &  \qquad \quad \ \ \      $08.81$   & \quad  $08.81$    \\
 CI &  &   \quad \ $ {\bf [08.80;08.99] }$ &\quad \ \ \  ${\bf [08.78;08.84] }$ &    & \quad \ \ \  ${\bf [08.80;08.82] }$  &    \\
&  $ 105$ & \qquad \quad  $10.61$ & \qquad \quad  \ \ \ $10.60$ &   \quad  $10.59$  & \qquad \quad  \ \  \ $10.59$ &  \quad  $10.59$     \\
CI &  & \quad \   $ {\bf [10.51;10.72] } $& \quad \  \ \    ${\bf [10.57;10.63] } $&        &  \quad \ \ \  ${\bf [10.58;10.60] } $ &   \\
&$ 110$ &  \qquad \quad  $12.53 $ &  \qquad \quad  \ \ \  $12.57$  &    \quad  $12.57$  & \qquad \quad  \ \ \  $12.57$   &  \quad  $12.57$   \\
 CI &   & \quad \ $  {\bf [12.42;12.64] }$ &  \quad  \ \ \    $  {\bf [12.53;12.60] }$ &      &   \quad  \ \ \    $  {\bf [12.56;12.58] }$  &    \\
& $ 115 $   &    \qquad \quad $ 14.72$   & \qquad \quad \ \  \ $14.74$ &   \quad  $14.75$     &  \qquad \quad \ \  \   $14.75$  & \quad   $ 14.75$     \\
  CI &   &  \quad \ ${\bf [14.60;14.84] }$ & \quad \ \ \   ${\bf [14.70;14.78] }$ &       &  \quad \ \ \   ${\bf [14.75;14.77] }$ &   \\ 
 & $ 120$      &  \qquad \quad  $17.18$   & \qquad \quad  \ \ \   $17.10$ &    \quad   $17.11$     & \qquad \quad  \ \ \   $17.13$ &  \quad  $17.12$    \\
  CI &   &  \quad \  ${\bf [17.04;17.31] }$ & \quad \ \ \    ${\bf [17.06;17.15] }$ &      &  \quad \ \ \    ${\bf [17.11;17.14] }$   &    \\
 & $125$     & \qquad \quad    $ 19.64 $   & \qquad \quad \ \  \ $19.69$  &  \quad  $19.66$   &  \qquad \quad  \ \ \    $19.67$ &   \quad  $ 19.67$ \\
  CI &  & \quad \    ${\bf [19.50;19.78] }$ & \quad \   \ \  ${\bf [19.64;19.73] }$ &       &   \quad \   \ \  ${\bf [19.65;19.68] }$ &    \\
   & $130$     & \qquad \quad   $ 22.41 $   & \qquad \quad   \ \  \  $22.32$ &     \quad   $22.39$     & \qquad \quad   \ \  \  $22.40$  & \quad $ 22.40$  \\
   CI &  & \quad \    ${\bf [22.26;22.56] }$ & \quad \   \ \  ${\bf [22.32;22.41] }$ &        & \quad \   \ \  ${\bf [22.38;22.41] }$ &    \\
  &   &      &    &  &     &   \\
 \hline
  \end{tabular} 
\caption{ \label{tabPrixPutLVStrike} \small{ (Pseudo-CEV model) Comparison of the Put prices obtained from Monte Carlo (MC) simulations (followed by its size) with associated confidence intervals (CI) and from the recursive marginal quantization (RMQ) method method  with  equal grid size  allocation $N_k=300$ or $N_k=400$, $\forall k=1,\ldots,n$. The parameters are:  $r =0.15$; $n=120$;  $N_k=400$, $\forall k=1,\ldots,n$; $T=1$; $\vartheta=4$; $X_{0}=100$;  and for varying values of $K$.} }
\end{center}
\end{table}

We aim at computing  the price  of a European Put option with payoff $(K- X_T)^{+}=\max(K- X_T,0)$, where $K$ corresponds to  the strike of the option  and $T$ to  its  maturity. Then we have to approximate the quantity
$$  e^{-rT} \mathbb E(K - X_T )^{+} $$
where $\mathbb E$ stands for the expectation under  the risk neutral probability.  If  the process  $(\bar X_{t_k})_k$ denotes the discrete Euler process  at regular  time discretization steps $t_k$, with $0= t_0<\cdots<t_n=T$, associated to the diffusion process  $(X_t)_{t \geq 0}$, this turns out to estimate 
$$  e^{-rT} \mathbb E(K - \bar X_T )^{+} $$
by optimal quantization.  We estimate this quantity by the recursive marginal quantization method introduced in this paper and compare the numerical results with those obtained from  standard  Monte Carlo simulations. 

\subsubsection{Numerical results}   \label{SecNumericalResults}

  To deal with numerical examples we set $\delta=0.5$, $X_0=100$, and choose the interest rate $r=0.15$. We discretize the price process using the Euler scheme with $n=120$ (regular)  discretization steps   and quantize the Euler marginal processes by our proposed  method.  For $k=1, \ldots,n$, we put all  the marginal quantization  grid   sizes  $N_k$ equal to $300$ and then, to $400$  (recall that $\widehat X_0 = X_0 =100$ and  $N_0=1$).  We estimate  the price of the Put option  by 
\begin{equation} \label{EqSumMQ}
   \mathbb E \big[ \big(K - \widehat X_{t_n}^{\Gamma_n}\big)^{+} \big] = \sum_{i=1}^{N_n} (K- x_i^{N_n})^{+} \, \mathbb P \big(\widehat X_{t_n}^{\Gamma_n} = x_i^{N_n} \big)  
 \end{equation}
where $t_n =T$, and where   $\Gamma_n = \{ x^{N_n}_1,\cdots, x^{N_n}_{N_n} \}$ is the  quantizer  of size $N_n$ computed from the  Newton-Raphson algorithm (with $5$ iterations) and where the associated weight are estimated from  (\ref{EqEstProba}).

   We compare the prices obtained from the recursive marginal quantization (RMQ) method with those obtained by the Monte Carlo (MC) simulations   for  various values of $\vartheta$ with a fixed strike $K=100$ (see Table~\ref{tabPrixPutLV}) or for varying  values of the strike $K$ with a fixed $\vartheta = 4$ (see Table~\ref{tabPrixPutLVStrike}). For the Monte Carlo simulations   we set  the  sample size  $M_{\rm mc}$ equal to  $10^{5}$ and   $10^{6}$ for $K=100$ and  to $M_{\rm mc}=10^{5}$,  $10^{6}$ and $10^{7}$ when making the strike $K$ varying.

\begin{rem} ({\em on the computation time})
$ (a)$ Remark that all the quantization grids $\Gamma_k$ of sizes $N_k =300$  (and $N_k=400$), for every $k=1,\ldots,n=120$, and there companion weights are obtained in  around $40$ seconds  (and $1$ minute)   from the  Newton-Raphson algorithm with $5$ iterations. Computations are performed using {\em Scilab} software  on a CPU  $2.7$  GHz  and   4 Go memory computer.   

\noindent $(b)$ It is clear that   once the grids and the associated  weights  are available, the estimation of the price by RMQ method using the sum (\ref{EqSumMQ}) is instantaneous.
\end{rem}

\begin{rem} ({\em Initialization of the Newton-Raphson  algorithm}) Let  $0= t_0<\cdots<t_n$ be the time discretization steps, let $X_0 = x$ be the present value of the stock price process and suppose that the grid sizes $N_k$ are all  equal. Since  the random variable $\bar X_{t_1} \sim \mathcal N(m_0(x);v_0^2(x))$, in order to  compute the (optimal) $N_1$-quantizer  for $\bar X_{t_1}$ we initialize the algorithm to $v_0(x) z^{N_1} + m_0(x)$, where $z^{N_1}$ is the optimal $N_1$-quantizer of the $\mathcal N(0;1)$. Once we get the optimal $N_1$-quantization $\Gamma_1$ for  $\bar X_{t_1}$ and its companion   weights,  we initialize the algorithm to $\Gamma_1$ to perform the optimal $N_{_2}$-quantizer for $\bar X_{t_{_2}}$ and its companion weights, $\cdots$, and so on, until we get the optimal $N_n$-quantizer for $\bar X_{t_n}$ and the associated weights. Notice   that doing so we observe  no failure of  the convergence in all the  considered examples.
\end{rem}

  \begin{rem} 
  We show in Figure~\ref{figure1} and Figure~\ref{figure2} two graphics  where we depict on the abscissa  axis   the optimal grids (of sizes $N_k = 150$) and on the ordinate axis the corresponding weights.  The dynamics of the price process in    Figure~\ref{figure1} is given by 
$$  dX_t = r X_t dt  + \sigma X_t dW_t, \quad X_0 = 86.3   $$
with $r=0.03$, $\sigma = 0.05$
whereas its dynamics  in Figure~\ref{figure2} is given by 
$$  dX_t = rX_t dt + \vartheta  \frac{X_t^{\delta+1}}{\sqrt{1+X_t^2}} dW_t, \quad X_{0} = 100$$ 
with $r=0.15$, $\vartheta = 0.7$, $\delta = 0.5$.  

\end{rem}

%
For our numerical examples, we remark first that  in all examples the prices obtained by RMQ (of sizes $N_k = 300$ and $N_k=400$, $\forall k=1, \cdots,n$) stay in the confidence interval induced by the MC   price estimates.  On the other hand the prices obtained by the RMQ (of sizes $N_k=400$) method are more precise (more especially  when $\vartheta=4$ and $K$ grows away from $100$) than  those obtained by the MC method when  $M_{\rm mc} = 10^{5}$ or $10^{6}$. Consequently,   the RMQ method seems to be  more efficient  than the MC when the sample size is less than $10^{6}$.   However, when increasing the sample size to $M_{\rm mc}=10^{7}$ the two prices become closer (up to $10^{-2}$).  We also remark that, up to $10^{-2}$, the prices obtained from RMQ of both different sizes ($N_k=300$ and $N_k=400$) are equal.

\begin{rem}
We remark that when the Monte Carlo sample  size  $M_{\rm mc}=10^{6}$ (resp. $M_{\rm mc}=10^{7}$) it takes about  $1$ minute and $40$ seconds (resp. $2$ minutes and $30$ seconds) to get a price using the {\em C programming language} on the same computer described previously. Then, in this situation,  it takes much more time    to obtain a price   by MC method than  by RMQ for a fixed precision in the price approximations. 
\end{rem}


To strengthen  the previous conclusions related to the local volatility model  we compare the two methods in the Black-Scholes  framework where the stock price process evolves following the dynamics:
$$ dX_{t} = r X_t dt + \sigma X_t dW_t, \quad  X_0 =100.$$
In this setting the true prices are available and will serve us as a support for  comparisons.  The  parameters are chosen so that  the model remains close to the Pseudo-CEV model: $r=0.15$ and $\sigma \approx \vartheta X_0^{\delta-1}$.   Numerical results are printed in Tables~\ref{tabPrixPutBS} and Table~\ref{tabPrixPutBSStrike} and confirm our conclusions  on the Pseudo-CEV model. We notice that in the Black-Scholes model, the estimated prices from the RMQ (for both size choice) method are  close to the true prices (the best absolute error is of order $10^{-5}$ for a volatility $\sigma =5\%$ and the worse absolute error $2. 10^{-2}$ is achieved with high volatility: $\sigma =40\%$).  This shows the robustness of the  RMQ method even for   reasonably high values of the volatility.

\vskip 0.3cm
 
\begin{table}[htbp]
 \begin{center}
 \footnotesize
  \begin{tabular}{l*{6}{l}}
&   &      &    &    &     \\
   \hline
  $\sigma$ & \qquad   MC ($10^5$) &  \   RMQ ($N_k=300$)  & \qquad MC ($10^6$)    &      \   RMQ ($N_k=400$)     & \ \ True price  \\
  \hline  
     &      &    &    &  &   \\
 $0.05 $ & \qquad  \ \  $0.0015 $ &  \quad $0.00178$  &  \qquad \ \ $ 0.00178$   & \quad  $0.00176$ &  \quad $0.00177$    \\
   &   ${\bf [0.0012;0.0019] }$ &\quad  ${\bf 1. 10^{-5}}$  & \quad  ${\bf [0.0017;0.0019] }$  & \quad    ${\bf 1. 10^{-5}}$  &     \\
  $ 0.06$ & \qquad \ \   $0.0116$ & \quad    $0.0108$ & \qquad \quad  $0.0109$ &  \quad    $0.0109 $ &\quad \ \ $0.0112 $     \\
  &  $ {\bf [0.0104;0.0128] } $&\quad   ${\bf 4.10^{-4}}$ & $ \quad  {\bf [0.0106;0.0113] } $  &\quad    ${\bf 3.10^{-4}}$  &    \\
$ 0.07$ &  \qquad \ \ $0.0365 $ & \quad   $0.0366$ &   \qquad \quad  $0.0370$  &  \quad    $0.0369$ &\quad  \ \ $0.0373$   \\
    & $  {\bf [0.0342;0.0387] }$ & \quad   ${\bf 7.10^{-4}}$  & $  \quad    {\bf [0.0363;0.0378] }$  & \quad   ${\bf 4.10^{-4}}$   &      \\
  $0.08$    &    \qquad \ \   $ 0.0876$   &  \quad  $0.0865$ & \qquad \quad $0.0876$   & \quad   $ 0.0869$ & \quad  \ \ $0.0875$      \\
     &   ${\bf [0.0836;0.0915] }$ &\quad   ${\bf  1.10^{-3}}$ & \quad   ${\bf [0.0863;0.0888] }$ &\quad   ${\bf  6.10^{-4}}$ &    \\ 
  $ 0.09$      &  \qquad \ \   $0.1666$  &   \quad   $0.1641$  &  \qquad \quad   $0.1644$ &   \quad   $0.1647$  & \quad  \ \ $0.1654$   \\
    &   ${\bf [0.1607;0.1724] }$ &\quad   ${\bf 1.10^{-3}}$ &  \quad  ${\bf [0.1622;0.1658] }$   & \quad   ${\bf 7.10^{-4}}$ &    \\
  $0.10 $    & \qquad \quad   $ 0.269 $   &   \quad  $ 0.270$   &  \qquad \quad \  \   $0.271$  &   \quad  $ 0.271$ &\qquad      $0.272$  \\
    & \ \ \     ${\bf [0.261;0.277] }$ & \quad   ${\bf 2.10^{-3}}$  & \quad  \ \ \    ${\bf [0.271;0.273] }$ &\quad   ${\bf 1.10^{-3}}$ &       \\
   $ 0.20$     & \qquad \quad  $ 2.444 $   &   \quad $ 2.422$  & \qquad \quad  \  \  $2.431$ & \quad $ 2.424$ &\qquad $2.427$  \\
   & \ \ \   ${\bf [2.410;2.479] }$ & \quad    ${\bf 5.10^{-3}}$ &  \quad \ \ \   ${\bf [2.420;2.442] }$   &\quad   ${\bf 3.10^{-3}}$ &   \\
    $ 0.30$     & \qquad \quad  $5.455$   &  \quad $ 5.466$ &  \qquad \quad   \ \  $5.469$  &  \quad $ 5.470$ &\qquad $5.474$ \\
    & \ \ \   ${\bf [5.395;5.515] }$ & \quad   ${\bf 8.10^{-3}}$  & \quad \ \ \  ${\bf [5.450;5.549] }$ &\quad    ${\bf 4.10^{-3}}$  &    \\
     $0.40$     & \qquad \quad   $8.680$   & \quad  $ 8.785$  &  \qquad \quad \  \    $8.787$  & \quad  $ 8.790$ &\qquad $8.792$  \\
    &   \ \ \  $ {\bf [8.598;8.763] }$ & \quad   ${\bf 7.10^{-3}}$  & \quad  \ \ \  ${\bf [8.760;8.813] }$  &\quad    ${\bf 2.10^{-3}}$   &   \\
 &    &      &    &    &   \\
 \hline
  \end{tabular} 
\caption{ \label{tabPrixPutBS} \small{ (Black-Scholes  model) Comparison of the Put prices obtained from Monte Carlo (MC) simulations (followed by its size) with associated confidence intervals  and from marginal quantization  (RMQ)   method with  equal grid size allocation  $N_k=300$ or $N_k=400$, $\forall k=1,\ldots,n$,  with the associated absolute error (under each  displayed RMQ price)   w. r. t. the true price. The parameters are:  $n=120$; $T=1$; $r =0.15$;  $K=100$; $X_{0}=100$  for varying values of $\sigma$.} }
\end{center}
\end{table}

\vskip 0.3cm

\begin{table}[htbp]
 \begin{center}
 \footnotesize
 \begin{tabular}{l*{6}{l}}
    
  \hline
  $K$ & \qquad  \   MC ($10^5$)  &   RMQ $(N_k=300)$ & \qquad \ \  MC ($10^6$)    &       RMQ $(N_k=400)$    &  \ \ True price  \\
  \hline  
     &      &    &    &  &   \\
 $100$  & \qquad \quad   $8.680$   &  \quad $8.785$  &  \qquad \quad \  \ \    $8.787$  & \quad $8.790$ &     \qquad $8.792$    \\
  &   \quad \  $ {\bf [8.598;8.763] }$ & \quad  ${\bf 7. 10^{-3}}$     & \quad \ \ \     ${\bf [8.760;8.813] }$  &   \quad  ${\bf 2. 10^{-3}}$      &   \\
  $ 105$ & \qquad \ \   $10.805$ &\quad  $10.740$    & \qquad \quad  \ $10.739$ &  \quad  $10.744$ &\quad \ \ $10.750 $   \\
  & \quad \  $ {\bf [10.71;10.90] } $& \quad  ${\bf 1.10^{-2}}$    &   \quad  \ \ \ ${\bf [10.71;10.90] } $  & \quad  ${\bf 6.10^{-3}}$    &   \\
$ 110$ &  \qquad \quad  $12.86 $ & \quad  $12.90$  &  \qquad \quad  \ \ \  $12.89$  &  \quad  $12.90$ &\qquad   $12.91$   \\
   & \quad \ $  {\bf [12.76;12.96] }$ & \quad   ${\bf 1.10^{-2}}$  &  \quad  \ \ \    $  {\bf [12.86;12.93] }$  &   \quad   ${\bf 1.10^{-2}}$  &  \\
 $ 115 $   &    \qquad \quad $ 15.29$   &  \quad   $ 15.25$  &  \qquad \quad \ \  \ $15.24$   & \quad   $ 15.26$ & \qquad   $15.27$       \\
    &  \quad \ ${\bf [15.18;15.40] }$ &   \quad  ${\bf  1.10^{-2}}$  &  \quad \ \ \   ${\bf [15.21;15.28] }$ & \quad   ${\bf  1.10^{-2}}$ &  \\ 
  $ 120$      &  \qquad \quad  $17.66$   &  \quad  $17.80$  &  \qquad \quad  \ \ \   $17.81$ &   \quad  $17.79$  & \qquad   $17.81$  \\
    &  \quad \  ${\bf [17.54;17.79] }$ &  \quad ${\bf 1.10^{-2}}$   &  \quad \ \ \    ${\bf [17.78;17.85] }$   & \quad   ${\bf 1.10^{-2}}$  &   \\
  $125$     & \qquad \quad    $ 20.56 $   &  \quad   $ 20.50$ & \qquad \quad \ \  \ $20.50$  &   \quad  $ 20.50$ &\qquad      $20.52$  \\
 & \quad \    ${\bf [20.43;20.69] }$ &   \quad ${\bf 2.10^{-2}}$  & \quad \   \ \  ${\bf [20.46;20.54] }$ &  \quad  ${\bf 1.10^{-2}}$   &  \\
    $130$     & \qquad \quad   $ 23.28 $   &  \quad  $23.37 $   & \qquad \quad   \ \  \  $23.37$ & \quad $ 23.37$ &\qquad $23.39$  \\
   & \quad \    ${\bf [23.14;23.42] }$ &  \quad  ${\bf 2.10^{-2}}$ & \quad \   \ \  ${\bf [23.34;23.43] }$ &  \quad   ${\bf 2.10^{-2}}$  &  \\
     &      &    &    & &   \\
 \hline
  \end{tabular} 
\caption{ \label{tabPrixPutBSStrike} \small{ (Black-Scholes  model) Comparison of the Put prices obtained from Monte Carlo (MC) simulations (followed by the size of the MC between  brackets) with associated confidence intervals  and from the marginal quantization  (RMQ)   method with equal grid size allocation  $N_k=300$ or $N_k=400$, $\forall k=1,\ldots,n$,  with the associated absolute error (under each  displayed RMQ price)  w. r. t. the true price. The parameters are:   $n=120$; $T=1$; $r =0.15$; $\sigma=40\%$; $X_{0}=100$ for varying values of $K$.} }
\end{center}
\end{table}

 \begin{figure}[htpb]
 \begin{center}
  \!\includegraphics[width=14.5cm,height=9.0cm]{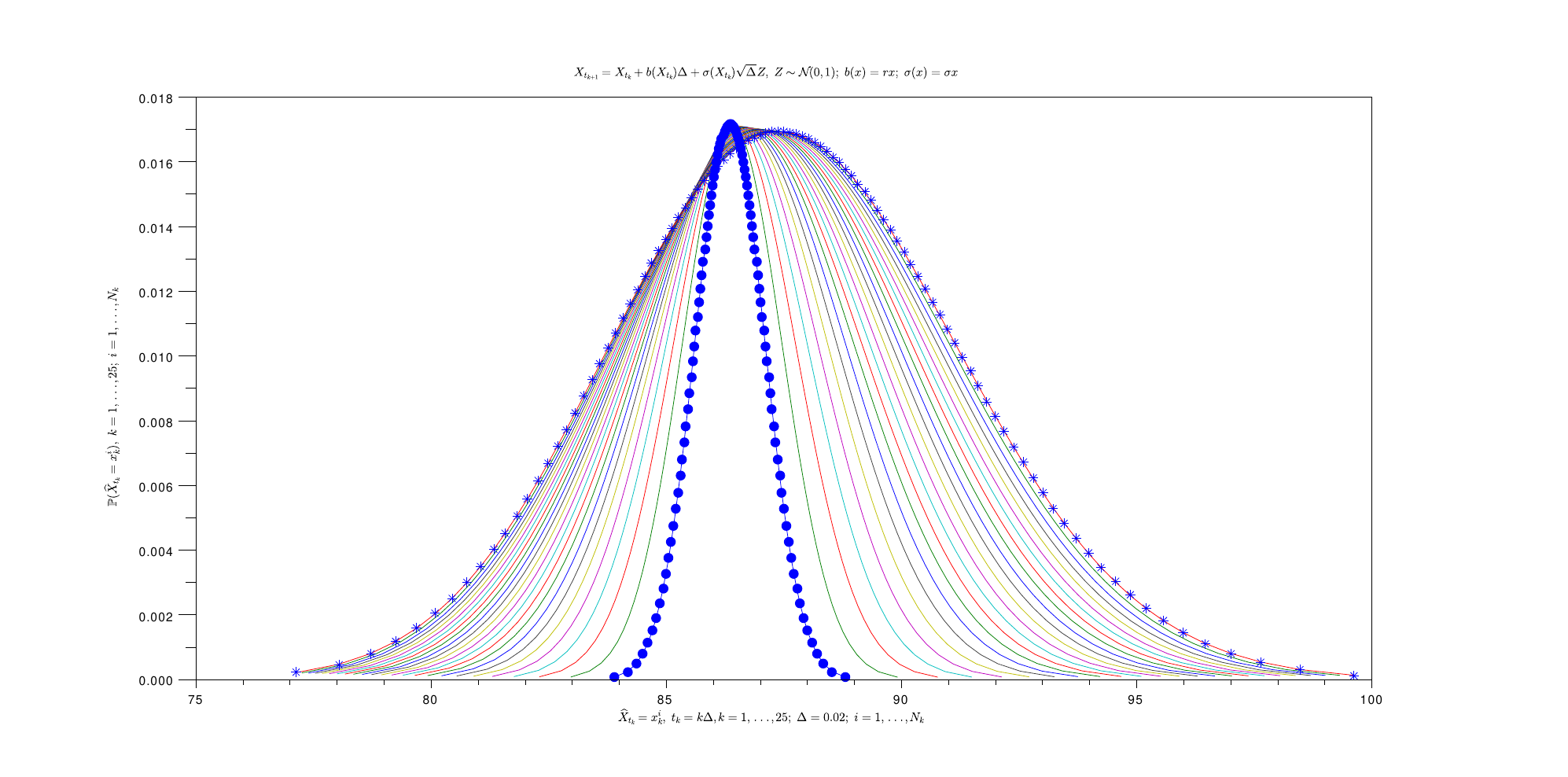}	
  \caption{\footnotesize  ("Black-Scholes model ") $d X_t = rX_t dt + \sigma X_t dW_t$, $X_0=86.3$, $r=0.03$, $\sigma = 0.05$. Abscissa axis: the optimal grids, $\widehat X_{t_k} = x_k^i$, $t_k = k\Delta$, $\Delta = 0.02$,  $k=1,\ldots,25$, $i=1,\ldots,N_k$. Ordinate axis: the associated weights, $\mathbb P(\widehat X_{t_k} = x^i_k)$, $k=1,\ldots,25$, $i = 1,\ldots,N_k$.  $\widehat{X}_{t_1}$ is depicted in  dots '$\bullet$', $\widehat X_{t_{25}}$ is represented by the symbol  '{\bf *}', $t_1 = 0.02$ and $t_{25} = 0.5$ and the remaining in continuous line} 
  \label{figure1}
  \end{center}
\end{figure}

 \begin{figure}[htpb]
  \begin{center}
    \!\includegraphics[width=14.5cm,height=9.0cm]{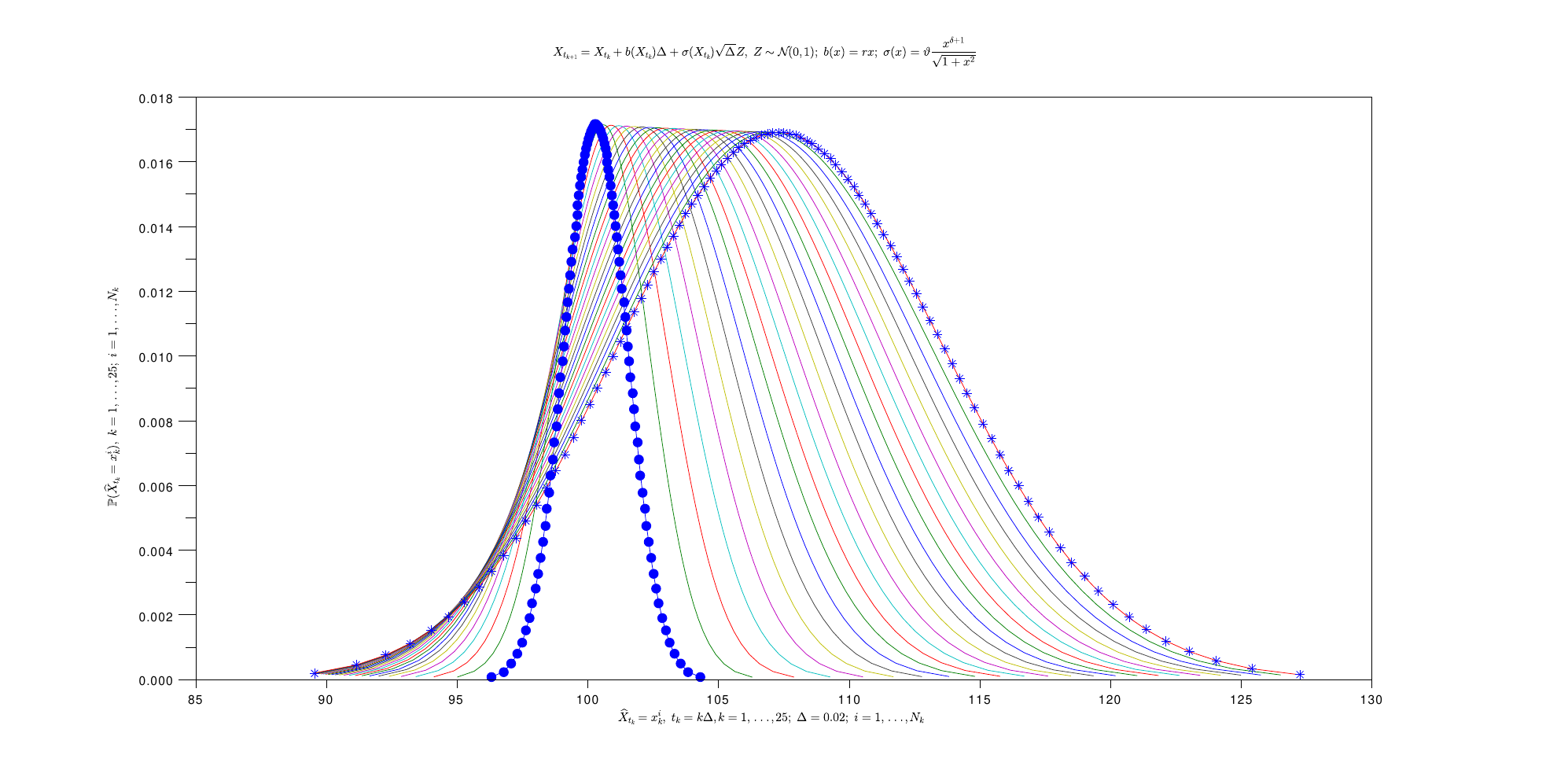}
  \caption{\footnotesize  ("Pseudo-CEV model") $dX_t = rX_t dt + \vartheta( X_t^{\delta+1}/ (1+X_t^2)^{-1/2} )dW_t$, $X_{0} = 100$, $r=0.15$, $\vartheta=0.7$, $\delta=0.5$. Abscissa axis: the optimal grids, $\widehat X_{t_k} = x_k^i$, $t_k = k\Delta$, $\Delta = 0.02$,  $k=1,\ldots,25$, $i=1,\ldots,N_k$. Ordinate axis:  the associated weights.  $\widehat{X}_{t_1}$ is depicted in  dots '$\bullet$', $\widehat X_{t_{25}}$ is represented by the symbol  '{\bf *}', $t_1 = 0.02$ and $t_{25} = 0.5$ and the remaining in continuous line.} 
  \label{figure2}
   \end{center}
\end{figure}

\newpage 
\begin{appendix}

\section*{Appendix}
The proof of Lemma~\ref{PropPrinc}  needs an additional result   given below as a technical Lemma.
\vskip 0.3cm

\noindent {\bf Lemma.}  Let $a\!\in \mathbb R^d$ et let $p\!\in [2,3]$.  Then, 
\begin{equation}   \label{LemPropPrinc}
\forall u\!\in \mathbb R^d,\quad \vert  a + u \vert^p  \le \vert a \vert^p + p \vert a \vert^{p-2} (a \vert u) + \frac{p(p-1)}{2} \big(\vert a \vert^{p-2} \vert u \vert^2 + \vert u \vert^p  \big) .
\end{equation}
 
 \vskip 0.3cm
 
  \begin{proof}[{\bf Proof.}] Define the function   $g(u) = \vert a+u  \vert^p$, $u\!\in \mathbb R^d$.  We have (denoting by $u^{\star}$ the transpose of the the column vector $u \!\in \mathbb R^d$),
  \[
  \nabla g(u) = p \vert  a + u \vert^{p-1}  \frac{a+u}{\vert a + u\vert} \quad \textrm{and} \quad  \nabla^2 g(u) = p (p-2) \vert  a + u \vert^{p-2} \frac{(a+u) (a+u)^{\star} }{\vert  a+u  \vert^2}  +p\vert  a + u \vert^{p-2}I_d  .
  \]
where $I_d$ is the identity on $\mathbb R^d$.  It follows from the Taylor-Lagrange formula that 
   \setlength\arraycolsep{3pt}
  \begin{equation*}
  \vert  a + u  \vert^p    =   \vert a  \vert^p + p \vert a  \vert^{p-2} (a \vert u) + \frac{p(p-2)}{2}    \vert a + \xi \vert^{p-2}  \frac{(a+\xi|u)^2}{\vert  a+\xi  \vert^2}   + \frac p2 \vert  a + \xi \vert^{p-2}  \vert u \vert^2,
  \end{equation*}
where $(\cdot \vert \cdot )$ stands for the inner product  and     $\xi  = \lambda_{u} u$,  $\lambda_u\!\in [0,1]$. Moreover $ \frac{(a+\xi|u)^2}{\vert  a+\xi  \vert^2}  \le   \vert u \vert^2 
$ owing to  Cauchy-Schwarz  inequality. Hence
   \begin{eqnarray*}
  \vert  a + u  \vert^p    &\le &    \vert a  \vert^p + p \vert a  \vert^{p-2} (a \vert u) + \frac{p(p-2)}{2}   \vert a + \xi \vert^{p-2}  \vert  u  \vert^2   + \frac p 2\vert  a + \xi \vert^{p-2}  \vert u \vert^2 \\
  & \le &  \vert a  \vert^p + p \vert a  \vert^{p-2} (a \vert u) + \frac{p(p-1)}{2}   \vert a + \xi \vert^{p-2}  \vert  u  \vert^2. 
  \end{eqnarray*}
Then, the result follows since  $ \vert a + \xi \vert^{p-2}  \le \vert a \vert^{p-2} + \vert u\vert^{p-2}$ since $p-2\!\in [0,1]$) and   $\vert \xi \vert \le \vert   u \vert$.
  \end{proof}
 We are now in position to prove Proposition~\ref{PropPrinc}.
 
 \begin{proof}[{\bf Proof.} {\it (of  Lemma~\ref{PropPrinc}.)}]  The proof will be split  into three  steps. 
 
\noindent  {\sc Step 1}.  Let $A$ be a $d {\small \times} q$-matrix.  We prove that for any random variable $Z$ such that  $\mathbb E (Z) =0$ and  $Z\!\in L^p(\Omega,{\cal A}, \mathbb P)$
\[
\mathbb E  \vert a + \sqrt{\Delta} A Z \vert^p   \le    \Big(1+  \frac{(p-1)(p-2)}{2}  \Delta \Big) \vert a \vert^p   +  \Delta \Big( 1 + p + \Delta^{\frac{p}{2}-1} \Big)  \Vert A  \Vert^p \mathbb E \vert  Z \vert^p,
\]
where   $\Vert A \Vert^2  = {\rm Tr} (A A^{\star})$.  In fact, it follows from Equation  ~\eqref{LemPropPrinc} that  
\[
\vert a + \sqrt{\Delta} A Z \vert^p  \le \vert a \vert^p + p  \Delta^{\frac{1}{2}} \vert a \vert^{p-2}  (a \vert A Z) + \frac{p(p-1)}{2}\big( \vert a \vert^{p-2}  \Delta  \vert  A Z \vert^2  + \Delta^{\frac{p}{2}}  \vert  A Z \vert^p  \big). 
\]
 Applying Young's inequality with conjugate  exponents $p' = \frac{p}{p-2}$ and $q' = \frac{p}{2}$, we get 
 \[
  \vert a \vert^{p-2}  \Delta  \vert  A Z \vert^2  \le \Delta \Big(  \frac{\vert a \vert^p}{p'} + \frac{\vert  A Z \vert^p}{q'} \Big),
 \]
 which leads to 
  \begin{eqnarray*}
 \vert a + \sqrt{\Delta} A Z \vert^p &  \le &  \vert a \vert^p + p  \Delta^{\frac{1}{2}} \vert a \vert^{p-2}  (a \vert A Z) + \frac{p(p-1)}{2}\Big(  \frac{\Delta}{p'}  \vert a \vert^{p} + \Big(\frac{\Delta}{q'} + \Delta^{\frac{p}{2}} \Big)   \vert  A Z \vert^p   \Big)  \\
 &  \le & \vert a \vert^p \Big( 1 +  \frac{p(p-1)}{2 p'} \Delta \Big) + p  \Delta^{\frac{1}{2}} \vert a \vert^{p-2}  (a \vert A Z)  +  \Delta \Big( \frac{p(p-1)}{2 q'} + \Delta^{\frac{p}{2}-1} \Big)   \vert  A Z \vert^p.
   \end{eqnarray*}
Taking the expectation    yields  (owing to the fact that $\mathbb E(Z) = 0$)
\[
\mathbb E  \vert a + \sqrt{\Delta} A Z \vert^p   \le     \Big( 1 +  \frac{(p-1)(p-2)}{2 } \Delta \Big)  \vert a \vert^p   +  \Delta \Big(1+ p + \Delta^{\frac{p}{2}-1} \Big)  \mathbb E \vert  A Z \vert^p.
\]
As a consequence, we get 
\[
\mathbb E  \vert a + \sqrt{\Delta} A Z \vert^p   \le    \Big( 1 +  \frac{(p-1)(p-2)}{2 } \Delta \Big)   \vert a \vert^p    +  \Delta \Big( 1+p + \Delta^{\frac{p}{2}-1} \Big)  \Vert A  \Vert^p \mathbb E \vert  Z \vert^p.
\]

\noindent {\sc Step 2}.  Keeping in mind the result of the first step and setting  for every $ t\!\in [0,T]$ and  $ x\!\in  \mathbb R^d$,  $a:=x+\Delta b(t, x)$ and   $A := \sigma(t,x)$, we get  (owing to the linear growth assumption on the coefficients of the diffusion process)
 \[
 \vert a \vert  \le  \vert  x \vert (1+L  \Delta ) +L \Delta   \quad \textrm{and} \quad  \Vert  A \Vert^p  \le L_p  (1 +\vert x \vert^p),
 \]
 where $L_p =2^{p-1} L^p$.   It follows that  (keep in mind that $p \!\in (2,3]$)
  \setlength\arraycolsep{3pt}
   \begin{eqnarray*} 
   \vert  a  \vert^p  & \le & (1+2 L \Delta )^p     \Big(   \frac{1+ L \Delta }{1 + 2 L \Delta } \vert x \vert   + \frac{L \Delta }{1 + 2 L \Delta } \Big)^p   \\
   & \le &   (1+2 L \Delta )^p     \Big(   \frac{1+ L \Delta }{1 + 2 L \Delta } \vert x \vert^{p}   + \frac{L \Delta }{1 + 2 L \Delta } \Big) \\
   & \le & (1+2 L \Delta )^p    \vert x \vert^{p}   +  (1+2 L \Delta )^{p-1}  L \Delta .
   \end{eqnarray*} 
 Then, we derive
  \begin{eqnarray*} 
   \mathbb E  \vert a + \sqrt{\Delta} A Z \vert^p  & \le &    \Big( 1 +  \frac{(p-1)(p-2)}{2 } \Delta \Big)    (1+ 2 L \Delta )^p \vert x \vert^p  \\
   & & +   \Big( 1 +  \frac{(p-1)(p-2)}{2 } \Delta \Big)    (1+ 2 L \Delta )^{p-1} L  \Delta  \\
   & &  +  \Delta L_p  \Big( 1 + p + \Delta^{\frac{p}{2}-1} \Big)  (1+\vert x \vert^p) \, \mathbb E \vert  Z \vert^p.
   \end{eqnarray*} 
Using the inequality $ 1+ u \le e^{u}$, for every $u\!\in \mathbb R$, we finally get 
 \begin{equation*} 
   \mathbb E  \vert a + \sqrt{\Delta} A Z \vert^p   \le     \Big( e^{ \kappa_p \Delta}  + K_p \Delta  \Big)  \vert x \vert^p  +  \big(  e^{\kappa_p  \Delta } L + K_p \big) \Delta, 
   \end{equation*} 
where $\kappa_p := \Big(\frac{(p-1)(p-2)}{2 } + 2 p L \Big)$  and $K_p := L_p \Big( 1 + p + \Delta^{\frac{p}{2}-1} \Big)   \mathbb E \vert  Z \vert^p$.

\noindent {\sc Step 3}.  Now, owing to the previous step and to the fact that for every $k=1, \ldots,n$,   $Z_{k}$ is independent from $\widehat X_{k-1}$,  we have 
 \begin{eqnarray*} 
 \mathbb E\vert  \widetilde X_k \vert^p   & = & \mathbb E \big[\mathbb E (\vert  {\cal E}_k(\widehat X_{k-1},Z_k) \vert^p \vert \widehat X_{k-1}) \big] \\
 & \le & \big( e^{ \kappa_p \Delta}  + K_p \Delta  \big)  \mathbb E  \vert \widehat X_{k-1} \vert^p  +  \big(  e^{\kappa_p  \Delta } L + K_p \big) \Delta.
  \end{eqnarray*} 
 Since  by construction, $\widehat X_{k}$ is a stationary quantizer (with respect to  $\widetilde X_{k}$) for every $k=0,\ldots,n$, we get    
   \begin{eqnarray*} 
 \mathbb E\vert  \widetilde X_k \vert^p   & = &    \big( e^{ \kappa_p \Delta}  + K_p \Delta  \big)  \mathbb E \big\vert \mathbb E ( \widetilde X_{k-1} \vert   \widehat X_{k-1}) \big\vert^p  +  \big( e^{\kappa_p  \Delta } L + K_p  \big)\Delta  \\
 & \le &   \big( e^{ \kappa_p \Delta}  + K_p \Delta  \big)  \mathbb E \big( \mathbb E ( \vert \widetilde X_{k-1} \vert^p  \vert   \widehat X_{k-1}) \big)  +  \big(e^{\kappa_p  \Delta } L + K_p \big) \Delta       \qquad (\textrm{Jensen's  inequality}) \\
 & =  &  \big( e^{ \kappa_p \Delta}  + K_p \Delta  \big)  \mathbb E \vert \widetilde X_{k-1} \vert^p   +  \big( e^{\kappa_p  \Delta } L + K_p  \big) \Delta .
  \end{eqnarray*} 
  We show by induction  that  for every $k=1,\ldots,n$,
  \begin{eqnarray*}
  \mathbb E\vert  \widetilde X_k \vert^p  &  \le&   \big( e^{ \kappa_p \Delta}  + K_p \Delta  \big)^k  \mathbb E \vert \widetilde X_{0} \vert^p   +   \big( e^{\kappa_p  \Delta} L + K_p \big) \Delta   \sum_{j=0}^{k-1} \big(  e^{ \kappa_p \Delta}  + K_p \Delta  \big)^j  \\
  & \le & e^{\kappa_p  k \Delta } \big( 1  + K_p \Delta e^{-\kappa_p \Delta }  \big)^k  \vert x_{0} \vert^p   +   \big( e^{\kappa_p  \Delta} L + K_p \big) \Delta   \sum_{j=0}^{k-1} e^{\kappa_p j \Delta }  \big( 1  + K_p \Delta e^{-\kappa_p \Delta }  \big)^j .
  \end{eqnarray*}
Using the inequality $1+u  \le e^u$, for every  $u\!\in \mathbb R$, yields 
 \begin{eqnarray*}
  \mathbb E\vert  \widetilde X_k \vert^p  & \le & e^{\kappa_p k \Delta } \big( 1  + K_p \Delta  \big)^k  \vert x_{0} \vert^p   +   \big(e^{\kappa_p \Delta } L + K_p \big) \Delta   \sum_{j=0}^{k-1} e^{\kappa_p j  \Delta }  \big( 1  + K_p \Delta   \big)^j \\
  & \le & e^{(\kappa_p  + K_p) k  \Delta }   \vert  x_{0} \vert^p   +   \big( e^{\kappa_p  \Delta }L + K_p \big) \Delta   \sum_{j=0}^{k-1} e^{(\kappa_p + K_p) j \Delta }  \\
  & = & e^{(\kappa_p  + K_p) t_k}   \vert  x_{0} \vert^p   +   \Delta  \big( e^{\kappa_p  \Delta } L + K_p \big)  \frac{e^{(\kappa_p + K_p) t_k}-1}{e^{(\kappa_p + K_p) \Delta} -1} \\
  & \le &  e^{(\kappa_p  + K_p) t_k}   \vert  x_{0} \vert^p   +     \big( e^{\kappa_p  \Delta } L + K_p \big) \frac{ e^{(\kappa_p + K_p) t_k}-1 }{ \kappa_p + K_p }.
  \end{eqnarray*}
The last inequality follows from the fact that $e^{(\kappa_p + K_p) \Delta} -1 \ge (\kappa_p + K_p)\Delta$.
  \end{proof}

\end{appendix}

\end{document}